\documentclass[preprint, 12pt]{elsarticle}

\usepackage{amssymb}
\usepackage{amsmath}
\usepackage{enumitem}
 \usepackage{amsthm}
  \usepackage{tikz -inet}
\usetikzlibrary{positioning,shapes.misc, patterns,arrows,decorations,decorations.pathreplacing, snakes}
\usepackage{graphicx}
\input xy
\xyoption{all}
\newtheorem{conjecture}{Conjecture}

\newtheorem{theorem}{Theorem}
\newtheorem{corollary}{Corollary}
\newtheorem{assumption}{Assumption}
\newdefinition{definition}{Definition}
\newtheorem{hypothesis}{Hypothesis}
\newdefinition{remark}{Remark}
\newtheorem{lemma}{Lemma}
\newtheorem{proposition}{Proposition}
\newtheorem{claim}{Claim}
\newtheorem{fact}{Fact}
\newcommand{\K}{\mathcal{K}}
\newcommand{\T}{\mathcal{T}}
\newcommand{\Union}{\bigcup}
\DeclareMathOperator{\tp}{ga-tp}
\DeclareMathOperator{\id}{id}
\DeclareMathOperator{\Aut}{Aut}
\DeclareMathOperator{\cf}{cf}
\DeclareMathOperator{\LS}{LS}
\DeclareMathOperator{\C}{\mathfrak{C}}
\DeclareMathOperator{\gaS}{gaS}
\newcommand*{\numberingI}[1]{%
\footnotesize\protect\tikz[baseline=-3pt]%
\protect\node[shape=circle,draw,inner sep=1.2pt,line width=0.2mm](n1){#1};}

\newcommand{\univ}{\rotatebox[origin=c]{-90}{$\prec^u$}}

\journal{Beyond First Order Logic}

\begin{document}

\begin{frontmatter}



\title{A Characterization of Uniqueness of Limit Models in Categorical Abstract Elementary Classes}


\author{Monica M. VanDieren\corref{cor1}}
\cortext[cor1]{Corresponding Author}
\address{Robert Morris University \\  6001 University Blvd \\ Moon Township PA 15108}
\ead{vandieren@rmu.edu}

\begin{abstract}
In this paper we examine the task set forth by Shelah and Villaveces in \cite{ShVi} of proving the uniqueness of limit models 
of cardinality $\mu$ in  $\lambda$-categorical abstract elementary classes with no maximal models, where $\lambda$ is some cardinal larger than $\mu$.  In \cite{Va} and \cite{Va-errata} we identified several gaps  in the approach outlined in  \cite{ShVi}, and we added the assumption that the union of an increasing chain of limit models is a limit model.

Here we replace this assumption with the seemingly weaker statement that the union of an increasing and continuous chain of limit models is an amalgamation base.
Moreover, we prove that this assumption  is not only sufficient but is necessary to settle the uniqueness of limit models  problem attempted in \cite{ShVi} for $\lambda=\mu^{+n}$ when $0<n<\omega$.

\end{abstract}




\end{frontmatter}



\section{Introduction}
Since its introduction in the 1970s, 
 the classification of abstract elementary classes (AECs) has been guided by generalizations of \L o\'{s} Conjecture \cite{Lo}:\begin{conjecture}[Shelah's Eventual Categoricity Conjecture \cite{Sh09b}]
If an AEC $\K$ is categorical in some sufficiently large cardinal $\lambda$, then $\K$ is categorical in all sufficiently large $\mu$.
\end{conjecture}

The amalgamation property seems to be key in proving categoricity transfer results related to this conjecture. In fact in 1986, Grossberg conjectured that the amalgamation property follows from categoricity \cite{Gr1}:
\begin{conjecture}\label{ap conj}
If $\K$ is an abstract elementary class categorical in a sufficiently large cardinality, then any triple of sufficiently large models from $\K$ can be amalgamated.
\end{conjecture}

Some progress has been made on this conjecture, but it remains open in general.   Kolman and Shelah  prove that amalgamation follows from categoricity in $L_{\kappa,\omega}$-axiomatizable AECs where $\kappa$ is a measurable cardinal \cite{KoSh}.  They first prove that the uniqueness of limit models follows from categoricity and then use this to derive the amalgamation property.
Another result is that categoricity in $\lambda\geq\beth_{\beth_{\omega_1}}$  in universal classes implies the amalgamation property
 by Vasey \cite[Corollary 2]{Vas2}.  For this result, Vasey derives the amalgamation property by working in an auxiliary class of models.  The catch is that this auxiliary class does not immediately appear to be an AEC; in particular, it may fail the smoothness property (see \cite[Definition 2.4.6(b)]{Vas2}).

 This paper 
uncovers a relationship between the two distinct approaches mentioned in the previous paragraph by uniting the property of the uniqueness of limit models with smoothness.  We  show that under some set-theoretic and model-theoretic assumptions that categoricity implies that a property which is related to smoothness of a class (see \ref{assumption item} in Theorem \ref{necessary theorem} below)  is equivalent to the uniqueness of limit models (see \ref{limit item} below).  The set-theoretic and model-theoretic assumptions are based on those originally identified by Shelah and Villaveces \cite{ShVi}.
\begin{theorem}\label{necessary theorem}
Let $\mu=\kappa^+$ be  cardinal so that $\LS(\K)\leq\kappa$.  Fix $n$ a natural number larger than zero and set $\lambda=\mu^{+n}$.
Suppose that GCH holds and assume $\Phi_{\chi^+}(S^{\chi^+}_{\cf(\chi)})$ for every $\chi$ satisfying $\kappa\leq\chi<\lambda$.
If $\K$ is $\lambda$-categorical and has no maximal models, then the following are equivalent:
\begin{enumerate}[label=\numberingI{\arabic*}]
\item\label{assumption item}  The union of an increasing and continuous chain of limit models $\langle M_i\in\K_\mu\mid i<\alpha<\mu^+\rangle$ is an amalgamation base (Assumption \ref{new assumption}).
\item\label{limit item} If $M$ and $M'$ are limit models of cardinality $\mu$ over $M_0$, then $M$ and $M'$ are isomorphic over $M_0$ ($M\cong_{M_0}M'$).
\item\label{limit not over item} If $M$ and $M'$ are limit models of cardinality $\mu$, then $M$ and $M'$ are isomorphic ($M\cong M'$).
\item\label{union item} The union of an increasing  chain of saturated models dense with $\kappa$-amalgamation bases $\langle M_i\in\K_{\mu}\mid i<\alpha<\mu^{+}\rangle$ is  saturated.
\end{enumerate}

\end{theorem}

Theorem \ref{necessary theorem} sheds light on a long-standing problem of deriving the uniqueness of limit models from categoricity in abstract elementary classes with no maximal models begun by Shelah and Villaveces \cite{ShVi}.  This is described in further detail in Section \ref{sec:structure of proof}.  Additionally, Theorem \ref{necessary theorem}  improves the main result of \cite{Va-errata} which is the implication $\ref{assumption item}\Rightarrow\ref{limit item}$ of Theorem \ref{necessary theorem} for $\lambda=\mu^+$. 

We begin with some preliminary definitions and results in the next section.
In Section \ref{sec:structure of proof} we summarize the literature on the uniqueness of limit models begun by Shelah and Villaveces
and outline the structure of Shelah and Villaveces' intended proof that categoricity implies the uniqueness of limit models
in \cite{ShVi}.  Section \ref{sec:preliminaries} explains how to negotiate saturated models when the amalgamation property is not assumed.
Then,
in Section \ref{sec:symmetry} we work on the implication $\ref{assumption item}\Rightarrow\ref{limit item}$ of Theorem \ref{necessary theorem}.  
We confirm that the error in the proof that reduced towers are continuous mentioned in \cite{Va-errata} can be addressed by proving $\mu$-symmetry.  
We verify that the 
proofs in the series of papers \cite{Va-sym} and \cite{Va-union} can be adapted  to this setting in which the full amalgamation property is not assumed.  From this we get not only the equivalence of $\mu$-symmetry and the statement that  reduced towers are continuous, but also  the fact that $\mu^+$-categoricity implies $\mu$-symmetry.
We then adopt \cite{Va-union}  to this setting to transfer the $\mu^{+(n-1)}$-symmetry down to  $\mu$.  Finally, Section \ref{sec:main theorem} contains the remainder of the proof of Theorem \ref{necessary theorem}.

We tackle the adaptation of the proofs from \cite{Va-union} and \cite{VV-transfer} in an upcoming paper which will be used to improve Theorem \ref{necessary theorem} by requiring only that $\mu<\lambda$.

\section{Background}\label{sec:background}

For the history of the literature surrounding the uniqueness of limit models and the preliminary definitions and notation (e.g. abstract elementary classes, Galois-types, stability, $\Phi_{\mu^+}(S^{\mu^+}_{\cf(\mu)})$, etc.), we refer the reader to \cite{Va}, \cite{GVV}, and \cite{BV-survey}.  Here we will review a few of the concepts that we use explicitly in the proof of Theorem \ref{necessary theorem}.

Although we will not have the full amalgamation property at our disposal in this paper, we do have enough amalgamation to carry out several arguments.  Here we recall the level of amalgamation that we are guaranteed in the context of Theorem \ref{necessary theorem}.

\begin{definition}
An \emph{amalgamation base} is a model $M\in\K_\mu$ for which any two models of cardinality $\mu$ extending $M$ in $\K$ can be amalgamated.  That is for every $M_1,M_2\in\K_\mu$ with $M\prec_{\K}M_1,M_2$, there is $M^*\in\K_\mu$ and $\K$-embeddings $f_1$ and $f_2$ so that the following diagram commutes:
\[
\xymatrix{\ar @{} [dr] M_1
\ar[r]_{f_1}  &M^* \\
M \ar[u]^{\id} \ar[r]_{\id}
& M_2 \ar[u]_{f_2}
}
\]
\end{definition}

The set-theoretic assumption $\Phi_{\mu^+}(S^{\mu^+}_{\cf(\mu)})$  along with categoricity above $\mu$ imply the density of amalgamation bases of cardinality $\mu$:

\begin{fact}[Theorem 1.2.5  of \cite{ShVi} or see Lemma 1.2.23 of \cite{Va}]
Suppose that  $\Phi_{\mu^+}(S^{\mu^+}_{\cf(\mu)})$ holds.
Assume that $\K$ is categorical in $\lambda$ and $\mu<\lambda$.

Then for every $M\in\K_\lambda$ and $N\prec_{\K}M$ of cardinality $\mu$, 
 there exists an amalgamation base $N'\in\K_\mu$ with $N\prec_{\K}N'\prec_{\K}M$.  
 
\end{fact}

\begin{definition}
For $\mu\geq\LS(\K)$ and $\theta$ a limit ordinal $<\mu^+$, we say that
$M\in\K_\mu$ is a \emph{$(\mu,\theta)$-limit model} if there exists  an increasing and continuous sequence of amalgamation bases $\langle M_i\in\K_\mu\mid i<\theta\rangle$ so that $M=\Union_{i<\theta}M_i$ and $M_{i+1}$ is universal over $M_i$.  In this case we say that $M$ is a $(\mu,\theta)$-limit model over $M_0$.  We also say $M$ is a \emph{limit model} if there is a limit ordinal $\theta<\mu^+$ for which $M$ is a $(\mu,\theta)$-limit model.
\end{definition}

In the context of Theorem \ref{necessary theorem}, limit models are amalgamation bases:
\begin{fact}[Fact 1.3.10 of \cite{ShVi} or Theorem 1.3.13 of \cite{Va} ]\label{limits are ab}
Suppose that $\K$ has no maximal models and is categorical in $\lambda$ and that $\mu$ is a cardinal with $\lambda>\mu\geq\LS(\K)$.  Assume that GCH holds.  Then any limit model of cardinality $\mu$ is an amalgamation base.  Additionally, for every amalgamation base $M\in\K_{\mu}$ and for every limit ordinal $\theta<\mu^+$, there exists a $(\mu,\theta)$-limit model $M'$ over $M$.
\end{fact}

By $\mu^+$-many repeated applications of Fact \ref{limits are ab}, for any amalgamation base $M\in\K_\mu$ we can find a $(\mu,\mu^+)$-limit model over $M$.  This model is saturated and will serve as a replacement for a monster model.  We will use  $\C$ to denote such a model in the following sections.

\begin{remark}
Note that if $M$ and $M'$ are $(\mu,\theta)$- and $(\mu,\theta')$-limit models, respectively, over $M_0$ and $\cf(\theta)=\cf(\theta')$, then by a back-and-forth construction, $M$ and $M'$ are isomorphic over $M_0$.  Therefore Claim \ref{main claim} is only interesting when $\cf(\theta)\neq\cf(\theta')$.
\end{remark}

Next we recall the definition of the dependence relation that we will be using throughout this paper: $\mu$-splitting.
\begin{definition}\label{mu-split defn}
For $M\in\K_\mu$ an amalgamation base and
$p\in \gaS(M)$, we say that \emph{$p$ $\mu$-splits over
$N$}
iff
$N\prec_{\K}M$ and there exist amalgamation bases $N_1,N_2\in\K_\mu$ and a
$\prec_{\K}$-mapping $h:N_1\cong N_2$ such that
\begin{enumerate}

\item $N\prec_{\K}N_1,N_2\prec_{\K}M$,
\item $h(p\restriction N_1)\neq p\restriction N_2$
 and
\item $h\restriction N= id_N$.
\end{enumerate}
\end{definition}

While $\mu$-splitting is not as versatile as forking, it does have the  extension and uniqueness properties:

\begin{fact}[Theorem I.4.10 of \cite{Va}]
Suppose that $M\in\K_\mu$ is an amalgamation base and universal over $N$ and $M'$ is an extension of $M$ of cardinality $\mu$ inside $\C$. 
If $\tp(a/M)$ does not $\mu$-split over $N$ and there exists $g\in\Aut_M(\C)$ so that $\tp(g(a)/M')$ does not $\mu$-split over $N$.
\end{fact}

\begin{fact}[Theorem I.4.12 of \cite{Va}]
Suppose that $N,M, M'\in\K_\mu$ are amalgamation bases with $M'$ universal over $M$ and $M$ universal over $N$.  If $p\in\gaS(M)$ does not $\mu$-split over $N$, then there exists a unique $p'\in\gaS(M')$ such that $p'$ extends $p$ and $p'$ does not $\mu$-split over $N$.
\end{fact}

The uniqueness of limit models is related to the statement that the union of saturated models is saturated, which in first order model theory is equivalent to superstability.  Therefore we will be considering $\mu$-superstable abstract elementary classes:
We will use the following definition of $\mu$-superstability:
\begin{definition}\label{ss defn}
$\K$ is \emph{$\mu$-superstable} if $\K$ is Galois-stable in $\mu$ and 
 $\mu$-splitting satisfies the property:
for all infinite $\alpha<\mu^+$, for every sequence $\langle M_i\mid i<\alpha\rangle$ of
  limit models of cardinality $\mu$ with $M_{i+1}$ universal over $M_i$, and for every $p\in\gaS(M_\alpha)$, where
  $M_\alpha=\bigcup_{i<\alpha}M_i$, we have that
there exists $i<\alpha$ such that $p$
does not $\mu$-split over $M_i$.

\end{definition}

\begin{remark}\label{ss remark}
Shelah and Villaveces show that under the assumptions of Theorem \ref{necessary theorem}, $\K$ is $\mu$-superstable \cite[Fact 2.1.3 and Theorem 2.2.1]{ShVi}.  Their proof  uses GCH, but in a non-essential way.  At the point that they use $2^{<\mu}=\mu$, the replacement of choosing minimal $\chi\leq\mu$ so that $2^{\chi}>\mu$ would be sufficient.  
\end{remark}

We will see that, in fact, a slightly stronger form of $\mu$-superstability follows from categoricity.  This stronger form of $\mu$-superstability is Definition \ref{ss defn} with the additional condition of $\mu$-symmetry.  The property of $\mu$-symmetry was introduced in \cite{Va-sym} and used to prove the uniqueness of limit models assuming the amalgamation property \cite{Va-union, VV-transfer}.  Here, we will adapt these proofs to the setting of \cite{ShVi} where the full amalgamation property is not assumed.

Before moving to the proof of Theorem \ref{necessary theorem}, we recall a fact about directed systems.   
The following is implicit in the proof of Theorem III.10.1 of \cite{Va}.  This fact is used to construct extensions of amalgamable towers in \cite{Va}. 
Key is the assumption that $\Union_{i<\theta}N_i$ is an amalgamation base.  Without this assumption, the direct limit may not lie in $\C$.  This was exactly the point in \cite{Va}
where an additional assumption was introduced  to resolve one of the issues with Shelah and Villaveces' proof of the uniqueness of limit models.  Here we show that a related assumption to the one in \cite{Va} is not only sufficient to derive the uniqueness of limit models but it is necessary.

\begin{fact}\label{direct limit lemma}
Suppose that $\theta$ is a limit ordinal and $\langle M_i\in\K_{\mu}\mid i<\theta\rangle$ and $\langle f_{i,j}\mid i\leq j<\theta\rangle$ form a directed system.  
Assume that each $M_i$ is an amalgamation base and that each $f_{i,j}$ can be extended to an automorphism of $\C$.
If $\theta$ is a limit ordinal $<\mu^+$ and $\langle N_i\mid i\leq\theta\rangle$ is an increasing and continuous sequence of amalgamation bases so that for every $i<\theta$, $N_i\prec_{\K}M_i$ and $f_{i,i+1}\restriction N_i=\id_{N_i}$, then there is a direct limit $M^*\prec_{\K}\C$ of the system  and $\K$-embeddings $\langle f_{i,\theta}\mid i<\theta\rangle$  so that 
\begin{enumerate}
\item each $f_{i,\theta}$ can be extended to an automorphism of $\C$
\item $\Union_{i<\theta}N_i\preceq_{\K}M^*$ and 
\item $f_{i,\theta}\restriction N_i=id_{N_i}$.
\end{enumerate}
\end{fact}


\section{Shelah and Villaveces' Approach to the Uniqueness of Limit Models}\label{sec:structure of proof}
Shelah and Villaveces endeavor to prove the uniqueness of limit models in categorical AECs with no maximal models \cite{ShVi}.  They use set-theoretic assumptions to derive the density of amalgamation bases from categoricity, and then they attempt to prove the uniqueness of limit models.  This property of the uniqueness of limit models is not only a stepping stone to both derive the amalgamation property \cite{KoSh} but is also used to prove categoricity transfer results (e.g. \cite{GV2, Sh394}).
Additionally, Shelah and Villaveces' work inspired several papers examining the uniqueness of limit models in non-categorical classes as a step to develop a classification theory for non-elementary classes 
\cite{GVV, Za, ViZa2, Dr, GB, Va-sym, Va-union}.
Despite this,  the main result stated in \cite{ShVi} remains open:
\begin{claim}[The main claim, Theorem 3.37, of \cite{ShVi}]\label{main claim}
Let $\mu$ and $\lambda$ be cardinals so that  $\LS(\K)\leq\mu<\lambda$.
Suppose that GCH 
and  $\Phi_{\mu^+}(S^{\mu^+}_{\cf(\mu)})$  hold.

If $\K$ is $\lambda$-categorical and has no maximal models, then if $M$ and $M'$ are limit models of cardinality $\mu$ over $M_0$, then $M$ and $M'$ are isomorphic over $M_0$.

\end{claim}

In this paper we continue the endeavor begun by Shelah and Villaveces which includes a long line of work spanning nearly twenty years: \cite{Va-thesis,Va,Va-errata,GVV, Va-sym, Va-union, VV-transfer}.  In many of these papers additional assumptions were added to derive the consequences of Claim \ref{main claim}.  
Here we identify an assumption that is not only sufficient, but is necessary, to prove a special case of Claim \ref{main claim}:
\begin{assumption}\label{new assumption}\footnote{This is not a global assumption in the paper.  It will be explicitly stated when used.  It is a restatement of \ref{assumption item} of Theorem \ref{necessary theorem}.}
The union of an increasing and continuous chain of limit models $\langle M_i\in\K_\mu\mid i<\alpha<\mu^+\rangle$ is an amalgamation base.
\end{assumption}

\begin{corollary}
 Assumption \ref{new assumption} is  necessary and sufficient to prove Claim \ref{main claim} when $\lambda=\mu^{+n}$ where $0<n<\omega$.
\end{corollary}

We stated Theorem \ref{necessary theorem} using the set-theoretic assumptions of \cite{ShVi}, plus additional instances of the weak diamond that are needed to work with limit models of different cardinalities.  However,  these set-theoretic assumptions can be  replaced with model-theoretic assumptions and/or eliminated:

\begin{remark}\label{GCH remark}
The assumptions of GCH and $\Phi_{\chi^+}(S^{\chi^+}_{\cf(\chi)})$ in Theorem \ref{necessary theorem} are used in three places:  
\begin{itemize}
\item GCH is in the proof of superstability \cite[Theorem 2.2.1]{ShVi} which we describe how to eliminate in Remark \ref{ss remark}.  
\item Another use of GCH is to get limit models of each cardinality.  
But if we do not have limit models of cardinality $\mu^+$, then the statement of the theorem is vacuously true.  The subtle point where we still use GCH is that if the theorem isn't vacuously true because we have limit models of cardinality $\mu^+$, then
we will still need to use limit models of cardinality $\mu$ to prove the theorem.  And, without assuming the full amalgamation property, it is unknown if $\mu$ stability and the existence of limit models of cardinality $\mu^+$ are enough to imply $\mu$ stability or the density of limit models of cardinality $\mu$.

\item
Finally,  the diamond-like property, $\Phi_{\chi^+}(S^{\chi^+}_{\cf(\chi)})$,  is used to show that limit models of cardinality $\chi$ are amalgamation bases.  While the conclusion of the theorem only involves models of cardinality $\mu=\kappa^+$, the proofs employ limit models of cardinality $\kappa$ and models of cardinality larger than $\mu$ but smaller than $\lambda$.  
\end{itemize}
\end{remark}

To prove Claim \ref{main claim} we show that for every pair of limit ordinals $\theta_1,\theta_2<\mu^+$, every 
$(\mu,\theta_1)$-model $M$ over $M_0$ can be written as a $(\mu,\theta_2)$ over $M_0$.  We outline the construction here, but more details on this  construction can be found in \cite{Va} and \cite{GVV}.  The idea is to build 
 an increasing and continuous array of models with $(\theta_1+1)$-rows and $(\theta_2+1)$-columns.  The $(\theta_1+1)^{st}$-row will be constructed to be relatively full (see Definition II.6.6 of \cite{Va}) and the union of  this relatively full sequence of models is a $(\mu,\theta_2)$-limit model.  We will also construct the array so that if $M^j_i$ is the model in the $j^{th}$ row and $\beta^{th}$ column of the array, then $M^{j+1}_\beta$ will be universal over $M^j_\beta$.  This will witness that the union of the last column of the array is a $(\mu,\theta_2)$-limit model.  See Figure \ref{fig:array}.

\begin{figure}[h]
\begin{tikzpicture}[scale =2.9,inner sep=.5mm]
\draw[rounded corners=5mm] (0,0) rectangle (3.8,.5);
\draw[rounded corners=5mm] (0,.5) rectangle (3.8,-1.75);
\draw (.85,.25) node {$M_0$};
\draw(1.25,.25) node {$M_1$};
\draw (1.75,.25) node {$\dots M_i$};
\draw (2.35,.25) node {$M_{i+1}$};
\draw (3.15,.2) node {$\dots M^0_{\theta_1}=\displaystyle{\Union_{k<\theta_1}M_k}$};
\draw[rounded corners=5mm] (0,.5) rectangle (3.8, -.4);
\draw (.85,-.15) node {$M^{1}_0$};
\draw (1.75,-.15) node {$\dots M^{1}_\beta$};
\draw (2.35,-.15) node {$M^{1}_{\beta+1}$};
\draw(1.25,-.15) node {$M^{1}_1$};
\draw(1.25,-.85) node {$\dots$};
\draw(1.25,-1.5) node {$\dots$};
\draw(1.25,-2) node {$\dots$};
\draw (3.15,-.2) node {$\dots  M^1_{\theta_1}=\displaystyle{\Union_{\gamma<\theta_1}M^{1}_\gamma}$};
\draw (.8,-.5) node {$\univ$};
\draw (1.8,-.5) node {$ \univ$};
\draw (2.3,-.5) node {$\univ$};
\draw (3.15,-.5) node {$\univ$};
\draw[rounded corners=5mm] (0,.5) rectangle (3.8, -1);
\draw (.85,-.85) node {$M^{j}_0$};
\draw (1.75,-.85) node {$M^{j}_\beta$};
\draw (2.35,-.85) node {$M^{j}_{\beta+1}$};
\draw (3.15,-.8) node {$\dots  M^j_{\theta_1}=\displaystyle{\Union_{\gamma<\theta_1}M^{j}_\gamma}$};
\draw[rounded corners=5mm] (0,.5) rectangle (1,-2.5);
\draw[rounded corners=5mm](0,.5) rectangle (1.5, -2.5);
\draw[rounded corners=5mm] (0,.5) rectangle (2.5, -2.5);
\draw[rounded corners=5mm] (0,.5) rectangle (2,-2.5);
\draw[rounded corners=5mm] (0,.5) rectangle (3.8, -2.5);
\draw (.8,-1.15) node {$\univ$};
\draw (.8,-1.5) node {$M^{j+1}_0$};
\draw (1.8,-1.15) node {$ \univ$};
\draw (1.8,-1.5) node {$ M^{j+1}_\beta$};
\draw (2.3,-1.15) node {$\univ$};
\draw (2.3,-1.5) node {$M^{j+1}_{\beta+1}$};
\draw (3.15,-1.15) node {$\univ$};
\draw (3.15,-1.55) node {$\dots  M^{j+1}_{\theta_1}=\displaystyle{\Union_{\gamma<\theta_1}M^{j+1}_\gamma}$};
\draw (.85,-2) node {$\vdots$};
\draw (1.75,-2) node {$\vdots$};
\draw (2.35,-2) node {$\vdots$};
\draw (3.2,-2) node {$\vdots$};
\draw (3.15,-2.25) node {$\displaystyle{\Union_{\gamma<\theta_1, i<\theta_2}M^i_\gamma=M^{\theta_2}_{\theta_1}}$};
\tikzset{
    position label/.style={
       below = 13pt,
       text height = 1.5ex,
       text depth = 1ex
    },
   brace/.style={
     decoration={brace, mirror,amplitude=10pt},
     decorate
   }
}
\node [position label] (cStart) at (0,-2.3){};
\node [position label] (cA) at (3.8,-2.3) {};
\draw [brace] (cStart.south) -- node [position label, pos=0.5] {Continuous relatively full tower of length $\theta_1+1$} (cA.south);
\end{tikzpicture}
\caption{The array of models demonstrating a $(\mu,\theta_1)$-limit model which is also a $(\mu,\theta_2)$-limit model.  The notation $M\prec^{u}N$ represents the statement that $M$ is universal over $N$.} \label{fig:array}
\end{figure}

We will view each row of the array as a tower.   A \emph{tower} is a sequence of length $\alpha$ of amalgamation bases (specifically limit models), denoted by $\bar M=\langle M_i\in\K_\mu\mid i<\alpha\rangle$, along with a sequence of designated elements $\bar a=\langle a_{i}\in M_{i+1}\backslash M_i\mid i+1<\alpha\rangle$, and a sequence of designated submodels $\bar N=\langle N_{i}\mid i+1<\alpha\rangle$ for which
 $M_i\prec_{\K}M_{i+1}$, $\tp(a_i/M_i)$ does not $\mu$-split over $N_i$, and $M_i$ is universal over $N_i$ (see Definition I.5.1 of \cite{Va}).  The class of all towers indexed by $\alpha$ containing models of cardinality $\mu$ is denoted by $\K^*_{\mu,\alpha}$.  When working with towers, 
we will use the notation $\T=(\bar M,\bar a,\bar N)\in\K^*_{\mu,\alpha}$ for towers of length $\alpha$  and other abbreviations from \cite{Va-sym} such as $(\bar M,\bar a,\bar N)\restriction\beta\in\K^*_{\mu,\beta}$ for the restriction of the tower $(\bar M,\bar a,\bar N)$ to index set $\beta$.   

 Notice that the sequence $\bar M$ in the definition of a tower is not required to be continuous.  In fact, many times we will not have continuous towers. It is exactly at the indices witnessing discontinuity that we might have a model that is not an amalgamation base over which we will need to amalgamate two extensions.   Also for $\alpha$ a limit ordinal, a continuous tower $\T\in\K^*_{\mu,\alpha}$ may still cause us issues if  the top of the tower, $\Union_{i<\alpha}M_i$, is not an amalgamation base.  To avoid these problems we will restrict ourselves to nice or amalgamable towers.  A tower   $\T\in\K^*_{\mu,\alpha}$ is \emph{nice} if for every limit $\beta<\alpha$, $\Union_{j<\beta}M_j$ is an amalgamation base.   
 A  tower $\T\in\K^*_{\mu,\alpha}$ is \emph{amalgamable} if it is nice and $\Union_{\gamma<\alpha}M_\gamma$ is an amalgamation base.  Trivially, under the assumption that limit models are amalgamation bases, continuous towers are nice, but they may not be amalgamable.  Also notice that under Assumption \ref{new assumption}, all towers are nice and amalgamable.

To make sure that in a given column the model in the $(i+1)^{st}$-row is universal over the model in the $i^{th}$-row, we consider the following definition of tower extensions:
\begin{definition}[Definition 3.6.3 of \cite{ShVi}]
For towers $(\bar M,\bar a,\bar N)$ and $(\bar M',\bar a',\bar N')$ in $\K^*_{\mu,\alpha}$, we say $$(\bar M,\bar a,\bar N)\leq (\bar M',\bar a',\bar N')$$ if $\bar a=\bar a'$, $\bar N=\bar N'$, $M_\beta\preceq_{\K}M'_\beta$, and whenever $M'_\beta$ is a proper extension of $M_\beta$, then $M'_\beta$ is universal over $M_\beta$.  If for each $\beta<\alpha$,  $M'_\beta $ is universal over $M_\beta$ we will write $(\bar M,\bar a,\bar N)< (\bar M',\bar a',\bar N')$.  We say that $\K^*_{\mu,\alpha}$ has the \emph{extension property} if  every $(\bar M,\bar a,\bar N)\in\K^*_{\mu,\alpha}$ has a $<$-extension in $\K^*_{\mu,\alpha}$.\end{definition}

In \cite{Va-thesis}, we 
notice that in order to get the extension property for towers, the argument outlined in \cite{ShVi} did not seem to converge,
but that a direct limit construction was sufficient.  In order to carry out the direct limit construction, however, we need to restrict ourselves to amalgamable towers \cite{Va}.  
\begin{fact}[Corollary III.10.6 of \cite{Va}]\label{extension for towers}
Under Assumption \ref{new assumption} and the context of Theorem \ref{necessary theorem}, for every amalgamable $\T\in\K^*_{\mu,\alpha}$ there exists $\T'\in\K^*_{\mu,\alpha}$ so that $\T<\T'$.
\end{fact}

Assumption \ref{new assumption} will give us the extension property for towers, but in order to complete the construction depicted in Figure \ref{fig:array} we will need to produce continuous extensions.  In particular we will need that the lower model in the figure $M^{\theta_2}_{\theta_1}$ is the union of the last row of the tower.  To get continuous extensions we will look at reduced towers.

\begin{definition}\label{reduced defn}\index{reduced towers}
A tower $(\bar M,\bar a,\bar N)\in\K^*_{\mu,\alpha}$ is said to 
be \emph{reduced} provided that for every $(\bar M',\bar a,\bar
N)\in\K^*_{\mu,\alpha}$ with
$(\bar M,\bar a,\bar N)\leq(\bar M',\bar a,\bar
N)$ we have that for every
$\beta<\alpha$,
$$(*)_\beta\quad M'_\beta\cap\Union_{\gamma<\alpha}M_\gamma = M_\beta.$$
\end{definition}

Once we have the extension property for towers (Fact \ref{extension for towers}) we are able to produce reduced towers using Fact \ref{direct limit lemma}:
\begin{fact}[Fact III.11.3 of \cite{Va}]\label{density of reduced}
Under  the context of Theorem \ref{necessary theorem}, for every amalgamable $\T\in\K^*_{\mu,\alpha}$ in $\C$ there exists $\T'\in\K^*_{\mu,\alpha}$ a reduced extension of $\T$ in $\C$.
\end{fact}

\begin{fact}[Lemma III.11.5 of \cite{Va}]\label{monotonicity of towers}
Under Assumption \ref{new assumption} and the context of Theorem \ref{necessary theorem} if $\T\in\K^*_{\mu,\alpha}$ is reduced, then for every $\beta<\alpha$, $\T\restriction\beta$ is reduced.

\end{fact}

\begin{fact}[Theorem III.11.2 of \cite{Va}]\label{union of reduced is reduced}
Under Assumption \ref{new assumption} and the context of Theorem \ref{necessary theorem} for $\theta$ a limit ordinal $<\mu^+$, if $\langle \T^i\in\K^*_{\mu,\alpha}\mid i<\theta\rangle$ is an $<$-increasing chain of continuous and reduced towers, then the union of this chain of towers is a continuous and reduced tower in $\K^*_{\mu,\alpha}$.
\end{fact}

Reduced towers are important because they can be shown to be continuous.
However, one of the gaps in \cite{ShVi} was in the proof that reduced towers are continuous.  This was resolved in \cite{Va-errata} for towers in $\K^*_{\mu,\alpha}$ if one assumes that $\K$ is categorical in $\mu^+$.  Later fixes appear in \cite{GVV} and \cite{VV-transfer} where one assumes the amalgamation property and additional model-theoretic assumptions.  In this paper we show the approach in \cite{VV-transfer} can be applied in our context with limited amalgamation.  Underlying the fix in \cite{VV-transfer} is the additional assumption of $\mu$-symmetry.  We restate the definition here introducing the nuance of amalgamation bases:
\begin{definition}[Definition 3 of \cite{Va-sym}]\label{sym defn}
We say that an abstract elementary class  exhibits \emph{$\mu$-symmetry} if  whenever models $M,M_0,N\in\K_\mu$ and elements $a$ and $b$  satisfy the conditions \ref{limit sym cond}-\ref{last} below, then there exists  $M^b$  a limit model over $M_0$, containing $b$, so that $\tp(a/M^b)$ does not $\mu$-split over $N$.  
\begin{enumerate} 
\item\label{limit sym cond} $M$ is an amalgamation base and universal over $M_0$ and $M_0$ is a limit model over $N$.
\item\label{a cond}  $a\in M\backslash M_0$.
\item\label{a non-split} $\tp(a/M_0)$ is non-algebraic and does not $\mu$-split over $N$.
\item\label{last} $\tp(b/M)$ is non-algebraic and does not $\mu$-split over $M_0$. 
   
\end{enumerate}

\end{definition}

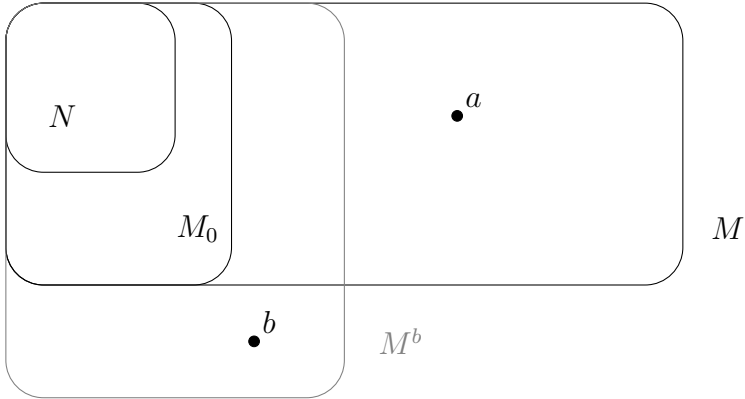
\begin{figure}[h]
\begin{tikzpicture}[rounded corners=5mm, scale=3,inner sep=.5mm]
\draw (0,1.25) rectangle (.75,.5);
\draw (.25,.75) node {$N$};
\draw (0,0) rectangle (3,1.25);
\draw (0,1.25) rectangle (1,0);
\draw (.85,.25) node {$M_0$};
\draw (3.2, .25) node {$M$};
\draw[color=gray] (0,1.25) rectangle (1.5, -.5);
\node at (1.1,-.25)[circle, fill, draw, label=45:$b$] {};
\node at (2,.75)[circle, fill, draw, label=45:$a$] {};
\draw[color=gray] (1.75,-.25) node {$M^{b}$};
\end{tikzpicture}
\caption{A diagram of the models and elements in the definition of symmetry. We assume the type $\tp(b/M)$ does not $\mu$-split over $M_0$ and $\tp(a/M_0)$ does not $\mu$-split over $N$.  Symmetry implies the existence of $M^b$ a limit model over $M_0$ containing $b$ so that $\tp(a/M^b)$  does not $\mu$-split over $N$.} \label{fig:sym}
\end{figure}

In \cite{Va-sym} under the assumption of the amalgamation property, this notion is shown to be equivalent to the statement that reduced towers are continuous --  the gap in the proof of Theorem 3.1.15 of \cite{ShVi} that is acknowledged and partially, but not completely, resolved in the errata \cite{Va-errata}.  VanDieren and Vasey show that for classes that satisfy the full amalgamation property, $\lambda$ categoricity implies $\mu$-symmetry for $\mu$ satisfying $\LS(\K)\leq\mu<\cf(\lambda)$ \cite[Corollary 7.2]{VV-transfer}.  In section \ref{sec:symmetry}, we verify that the arguments from \cite{Va-sym} and  \cite{VV-transfer} can be carried out in this context, thereby fully resolving the problem described in \cite{Va-errata}.  This will show the  implication $\ref{assumption item}\Rightarrow\ref{limit item}$ of Theorem \ref{necessary theorem}.

\section{Limit and Saturated Models}\label{sec:preliminaries}
In this section we verify some basic facts about saturated models in the context of Theorem \ref{necessary theorem} where only a limited amount of amalgamation is assumed.
In this section we make the following assumptions which follow from the assumptions of Theorem \ref{necessary theorem}:
\begin{hypothesis}\label{hyp:limit}
We assume that $\K$ is an abstract elementary class satisfying the following conditions for a fixed $\kappa$ with $\LS(\K)\leq\kappa<\lambda$:
\begin{enumerate}
\item Density of amalgamation bases of cardinality $\kappa$ and $\kappa^+$.
\item Limit models of cardinality $\kappa$ and $\kappa^+$ are amalgamation bases.
\item For $\chi=\kappa$ and $\kappa^+$, for every limit ordinal $\theta<\chi^+$ and every amalgamation base $N\in\K_\chi$ there exists $M\in\K$ a $(\chi,\theta)$-limit model extending $N$.
\end{enumerate}

\end{hypothesis}

Because we do not have the full amalgamation property, it may be the case that there are two non-isomorphic Galois-saturated models of cardinality $\kappa^+$ in our context.  For instance we might have a Galois-saturated model of cardinality $\kappa^+$ that is trivially saturated by way of having no or few submodels of cardinality $\kappa$ that are amalgamation bases.  
Alternatively, we might have two saturated models: one which is an amalgamation base and one which is not.
Fortunately we can avoid these kinds of anomalies in our proofs in later sections by restricting ourselves to saturated models which are dense with amalgamation bases.

\begin{definition}
A model $M$ of cardinality $>\kappa$ is said to be \emph{dense with $\kappa$-amalgamation bases} if for every $N\prec_{\K}M$ of cardinality $\kappa$ there exists an amalgamation base
$N'\in\K_\kappa$ for which $N\prec_{\K}N'\prec_{\K}M$.

\end{definition}


\begin{lemma}\label{sat universal lemma}

Suppose that $M$ is a saturated model of cardinality $\kappa^+$ that is dense with $\kappa$-amalgamation bases.  Then $M$ is  universal over $N$ for every amalgamation base $N\prec_{\K}M$ 
 of cardinality $\kappa$.
\end{lemma}

Notice that we do not require that $M$ be an amalgamation base at this stage; however, later, in Corollary \ref{sat is limit} this is established.

\begin{proof}
The proof is an adaptation of the proof that saturated models are model homogeneous which assumes the full amalgamation property \cite[Theorem 2.12]{GV}.   Let $M^*$ be a $(\kappa,\kappa^+)$-limit model extending $M$ which is also universal over $N$.  
We will use $M^*$ as a replacement for a monster model.

Fix $N'$ a model of cardinality $\kappa$ so that $N\prec_{\K}N'\prec_{\K}M^*$.  Let $\langle a_i\mid i<\kappa\rangle $ be an enumeration of $N'\backslash N$.  By induction on $i<\kappa$ we will define increasing and continuous sequences of models $\langle N'_i\mid i<\kappa\rangle$ and $\langle N_i\mid i<\kappa\rangle$ and mappings $\langle f_i\mid i<\kappa\rangle $
and $\langle f'_i\mid i<\kappa\rangle$
 so
that the following properties are satisfied:
\begin{enumerate}
\item $N_i$ is a  model of cardinality $\kappa$ (note that we do not require $N_i$ to be an amalgamation base.)
\item\label{ab condition} $N'_i$ is an amalgamation base of cardinality $\kappa$.
\item $N_i\prec_{\K}N'_i\prec_{\K}M^*$.
\item $N_0=N$ and $N'_0=N'$.
\item $a_i\in N_{i+1}$.
\item\label{univ over N'} either $N'_{j+1}=N'_j$ or $N'_{j+1}$ is universal over $N'_j$
\item $f_i:N_i\rightarrow M$ with $f_0=id_N$.
\item $f'_i:N'_i\rightarrow M^*$ with $f_i\subseteq f'_i$.
\end{enumerate}

Clearly this construction is sufficient since $\bigcup_{i<\kappa}f_i\restriction N'$ is as required.  

The only issue that needs to be checked at limit stages is that $N'_j$ is an amalgamation base, but this is guaranteed by conditions \ref{ab condition} and \ref{univ over N'} of the construction and 
Hypothesis \ref{hyp:limit}.

Let us consider the successor stage: $i=j+1$.  
Suppose that $f_j$, $f'_j$, $N_j$, and $N'_j$ have been defined.
If $a_j\in N_j$, let $N_{j+1}:=N_j$, $N'_{j+1}:=N'_j$,  $f_{j+1}:=f_j$, and $f'_{j+1}:=f'_j$.  So suppose that $a_j\notin N_j$.  
Let $M_j:=f_j[N_j]$ and $M'_j:=f'_j[N'_j]$.  Notice that the diagram below commutes:
\[
\xymatrix{\ar @{} [dr] N'_j
\ar[r]_{f'_j}  &M'_j \\
N_j \ar[u]^{\id} \ar[r]_{f_j}
& M_j \ar[u]_{id}
}
\]

Since $a_j\in N'_j\backslash N_j$, $f'_j(a_j)\in M'_j\backslash M_j$.  There are two cases to consider: $f'_j(a_j)\in M$ and $f'_j(a_j)\notin M$.
  If $f'_j(a_j)\in M$, since $N'_j$ is an amalgamation base, we can find $\bar{f_j}$ an automorphism of $M^*$ extending $f'_j$.  Let $M_{j+1}$ be a submodel of $M$ of cardinality $\kappa$ extending $M_j$ and $f'_j(a_j)$.  Let $N_{j+1}:=\bar f_{j}^{-1}[M_j]$.  Then let $N'_{j+1}$ be an amalgamation base of cardinality $\kappa$ which is a universal extension over $N'_j$, contains $N_{j+1}$, and lies inside $M^*$.  Then $f_{j+1}:=\bar f_j\restriction N_{j+1}$ and $f'_{j+1}:=\bar f_{j}\restriction N'_{j+1}$ are as required.

For the other case suppose that  $f'_j(a_j)\notin M$.  As before set $M_j:=f_j[N_j]$.  Since $M$ is dense with amalgamation bases there exists $\hat M_j\prec_{\K}M$ extending $M_j$ which is an amalgamation base of cardinality $\kappa$.  We can then consider the non-algebraic type, $p:=\tp(f'_j(a_j)/\hat M_j)$.  Because $M$ is saturated there exists $b\in M$ realizing $p$.  Let $\hat M^b$ 
be an amalgamation base of cardinality $\kappa$ inside $M$ containing $b$ and extending $\hat M_j$.  By the definition of equality of types, we can find $h\in \Aut_{\hat M_j}M^*$ so that $h(b)=a$ and the following diagram commutes:

\[
\xymatrix{\ar @{} [dr] M'_j
\ar[r]_{id}  &M^* \\
\hat M_j \ar[u]^{\id} \ar[r]_{id}
& \hat M^b\ar[u]_{h}
}
\]

We can replace $M^*$ in the diagram with some submodel $\hat M$ of cardinality $\kappa$ containing $h[\hat M^b]$ and universal over $M'_j$.  This is possible since $M'_j$ is isomorphic to $N'_j$ which was chosen to be an amalgamation base.
Then gluing this diagram together with the previous diagram gives us

\[
\xymatrix{\ar @{} [dr] N'_j
\ar[r]_{f'_j}  &M'_j\ar[r]_{id} &\hat M \\
N_j \ar[u]^{\id} \ar[r]_{f_j}
& M_j \ar[u]_{id} \ar[r]_{id} &\hat M^b \ar[u]_{h} \ar[r]_{id}&M
}
\]

Let $N'_{j+1}:=\bar f_{j}^{-1}[\hat M]$ and set $N_{j+1}:=\bar f_j^{-1}(h[\hat M^b])$.  
\[
\xymatrix{
\ar @{} [dr] N'_{j+1}\ar[drr]^{\bar f_j}\\
N'_j \ar[u]_{id}
\ar[r]_{f'_j}  &M'_j\ar[r]_{id} &\hat M \\
N_j \ar[u]^{\id} \ar[r]_{f_j}
& M_j \ar[u]_{id} \ar[r]_{id} &\hat M^b \ar[u]_{h} \ar[r]_{id}&M
}
\]

Then $f'_{j+1}:=\bar f_j\restriction N'_{j+1}$ and $f_{j+1}:=(h^{-1}\circ \bar f_j)\restriction N_{j+1}$ are as required.

\end{proof}

Notice that $(\kappa^+,\kappa^+)$-limit models and $(\kappa,\kappa^+)$-limit models are isomorphic:   

\begin{proposition}\label{unique saturated proposition}

If $M$ is a $(\kappa,\kappa^+)$-limit model over $N$ and $M'$ is a $(\kappa^+,\kappa^+)$-limit model  over some $M'_0$ containing $N$, them $M$ and $M'$ are isomorphic over $N$.
\end{proposition}

\begin{proof}
Let $\langle M_i\in\K_\kappa\mid i<\kappa^+\rangle$ witness that $M$ is a  $(\kappa,\kappa^+)$-limit model with  $N=M_0$ and
let $\langle M'_i\in\K_{\kappa^+}\mid i<\kappa^+\rangle$ witness that $M'$ is a  $(\kappa^+,\kappa^+)$-limit model with $N\prec_{\K}M'_0$.  Fix $\langle a'_i\mid i<\kappa^+\rangle$  an enumeration of $M'$.

 Since the models $M_i$ in the resolution of $M$ are all amalgamation bases, we are able to carry out the standard construction of  an isomorphism $f:M\cong M'$ by an increasing and continuous sequence of partial mappings
$f_i:M_i\rightarrow M'$ so that
$f_0$ is the identity mapping and
$a'_i\in f_{i+1}[M_{i+1}]$.
\end{proof}

Proposition \ref{unique saturated proposition} along with the following corollaries are used in the proof of $\ref{union item}\Rightarrow\ref{assumption item}$ of Theorem \ref{necessary theorem}.  A subtlety here is that the standard proofs of the uniqueness of saturated models require that the models are dense with amalgamation bases.

\begin{corollary}\label{sat is limit}
If $M$ is a saturated model of cardinality $\kappa^+$ that is dense with $\kappa$-amalgamation bases, then $M$ is a $(\kappa,\kappa^+)$-limit model.
\end{corollary}

\begin{proof}
Similar to the proof of Proposition \ref{unique saturated proposition}.
\end{proof}

Furthermore,  we will need to show that the saturated model that we construct is in fact an amalgamation base.  This follows from Proposition \ref{unique saturated proposition}, Corollary \ref{sat is limit}, and Fact \ref{limits are ab}.
\begin{corollary}\label{sat is ab}
If $M$ is a saturated model of cardinality $\kappa^+$ that is dense with $\kappa$-amalgamation bases, then $M$ is an amalgamation base.
\end{corollary}

How might we construct models that are dense with amalgamation bases?  First notice that  $(\kappa,\kappa^+)$-limits are trivially dense with amalgamation bases of cardinality $\kappa$.  
Thus by Proposition \ref{unique saturated proposition}, $(\kappa^+,\kappa^+)$-limit models are also dense with amalgamation bases of cardinality $\kappa$.
This allows us to show that any limit model is dense with amalgamation bases:
\begin{lemma}\label{limits are dense with ab}
For $\theta$ a limit ordinal $<\kappa^{++}$,
if $M$ is a $(\kappa^+,\theta)$-limit model, then $M$ is dense with amalgamation bases of cardinality $\kappa$.
\end{lemma}
\begin{proof}
Let $M$ be a  $(\kappa^+,\theta)$-limit model.  By the uniqueness of limit models of the same cofinality we may assume that $M=\Union_{i<\theta}M_i$ where $\langle M_i\mid i<\theta\rangle$ is an increasing and continuous sequence of amalgamation bases of cardinality $\kappa^+$ so that $M_{i+1}$ is a $(\kappa^+,\kappa^+)$-limit model over $M_i$.  Then by Proposition \ref{unique saturated proposition}, we know that each for successor $i$,  $M_i$ can be viewed as a $(\kappa,\kappa^+)$-limit model. For each successor $i<\theta$, let $\langle M^\alpha_i\in\K_\kappa\mid \alpha<\kappa^+\rangle$ witness that $M_i$ is a $(\kappa,\kappa^+)$-limit model.  

  Let $N\prec_{\K}M$ be a submodel of cardinality $\kappa$.  We need to find an amalgamation base $N'$ of cardinality $\kappa$ extending $N$ inside $M$.  
Without loss of generality, by renumbering if necessary, we may assume that $M^0_i\supseteq N\bigcap M_i$.

Define by induction on $i<\theta$ an increasing and continuous sequence $\langle N'_i\mid i <\theta\rangle$ of amalgamation bases of cardinality $\kappa$ so that $N'_{i+1}$ is universal over $N'_i$, $N'_i\prec_{\K}M_i$, and $N\bigcap M_i\subseteq N'_i$.  Let $N'_0:=M^0_0$.  At limit stages $i$, set $N'_i:=\Union_{j<i}N'_j$.  Notice $N'_i$ is a limit model by our inductive construction.  And, hence, it is an amalgamation base.  Now for the successor stage of the construction $i=j+1$, assume that $N'_j$ has been defined.  Since $N'_j$ has cardinality $\kappa$, we know that there exists $\alpha<\kappa^+$ so that $N'_j\prec_{\K}M^\alpha_{j+1}$.  Take $N'_{j+1}:=M^{\alpha+1}_{j+1}$.

Notice that $N':=\Union_{i<\theta}N'_i$ is a $(\kappa,\theta)$-limit model inside $M$ and extends $N$.  Since limit models are amalgamation bases, we are done.
\end{proof}

The following will also be used in the proof of $\ref{union item}\Rightarrow\ref{assumption item}$ of Theorem \ref{necessary theorem}.
\begin{lemma}\label{union is dense}
Suppose that $\theta$ is a limit ordinal $<\kappa^{++}$.
If $\langle M_i\mid i<\theta\rangle$ is an increasing and continuous sequence of saturated models of cardinality $\kappa^+$ and each is dense with $\kappa$ amalgamation bases, then 
$M:=\Union_{i<\theta}M_i$ is dense with $\kappa$-amalgamation bases.
\end{lemma}

\begin{proof}
To see that $M$ is dense with amalgamation bases, let $N\prec_{\K}M$ have cardinality $\kappa$.  If there exists $i<\theta$ so that $N\prec_{\K}M_i$ then we are done since by our assumption, $M_i$ is dense with $\kappa$ amalgamation bases so there is $N'\prec_{\K}M_i\prec_{\K}M$ an amalgamation base of cardinality $\kappa$ extending $N$ as required.

So suppose that for each $i<\theta$, $N\bigcap M_i\neq N$.    Because each $M_i$ is saturated and dense with amalgamation bases, by Corollary \ref{sat is limit} each $M_i$ is a $(\kappa,\kappa^+)$-limit model.  This allows us to 
 construct an increasing and continuous sequence of amalgamation bases of cardinality $\kappa$,
$\langle N_i\mid i<\theta\rangle $, so that $N\bigcap M_i\prec_{\K}N_i\prec_{\K}M_i$ and $N_{i+1}$ is universal over $N_i$.  Notice that $\Union_{i<\theta}N_i$ lies in $M$, extends $N$, and is a limit model and hence an amalgamation base.

\end{proof}

Note that in Lemma \ref{union is dense} we cannot conclude outright that $\Union_{i<\theta}M_i$ is also saturated without assuming some superstability.

\section{Symmetry and reduced towers}\label{sec:symmetry}
In this section we discuss the connection between the uniqueness of limit models and $\mu$-symmetry.  This is used to prove $\ref{assumption item}\Rightarrow\ref{limit item}$ of Theorem \ref{necessary theorem}.

For this section we make the following hypothesis:
\begin{hypothesis}\label{hypothesis symmetry section}
We assume that $\K$ is an abstract elementary class satisfying the following conditions for every $\kappa$ with $\LS(\K)\leq\kappa<\lambda$,
\begin{enumerate}

\item Limit models of cardinality $\kappa$ are amalgamation bases.
\item\label{limits exist hyp} For every limit ordinal $\theta<\kappa^+$ and every amalgamation base $N\in\K_\kappa$ there exists $M\in\K$ a $(\kappa,\theta)$-limit model over $N$.
\item $\K$ is $\kappa$-superstable.
\item The union of an increasing chain of limit models of cardinality $\kappa$ is an amalgamation base (Assumption \ref{new assumption}).
\end{enumerate}
\end{hypothesis}

We show that the arguments from \cite{Va-sym} and \cite{VV-transfer} can be carried out without the amalgamation property, if we assume only Hypothesis \ref{hypothesis symmetry section},  
to
prove
that reduced towers are continuous in categorical classes:
\begin{theorem}\label{reduced are continuous}
Suppose that Hypothesis \ref{hypothesis symmetry section} holds.
Suppose that  $\lambda$ and $\mu$ are cardinals so that there exists $0<n<\omega$ so that $\LS(\K)\leq\mu<\mu^{+n}=\lambda$. 
If $\K$ is categorical in $\lambda$, then reduced towers in $\K^*_{\mu,\alpha}$ are continuous if  $\alpha<\mu^+$.
\end{theorem}

\begin{proof}
When we take $n=1$, Theorem \ref{reduced are continuous} reduces to Theorem 2 of \cite{Va-errata}.
The case $n>1$ of Theorem \ref{reduced are continuous} is proved by first showing that $\mu$-symmetry implies that reduced towers are continuous (Theorem \ref{symmetry reduced}) and then deriving   $\mu$-symmetry from categoricity in $\lambda=\mu^{+n}$ (Theorem \ref{categoricity implies symmetry}).  
\end{proof}

The remainder of the section is dedicated to prove
the two theorems referenced in the proof of Theorem \ref{reduced are continuous}.
 In order to prove Theorem \ref{categoricity implies symmetry}, we need the converse of Theorem \ref{symmetry reduced}.  
We begin by establishing the equivalence of $\mu$-symmetry and the statement that reduced towers of cardinality $\mu$ are continuous (Theorem \ref{symmetry reduced} and its converse Theorem \ref{converse}).  
Then we finish the section by proving that $\mu$-symmetry can be derived from categoricity in $\mu^{+n}$ for some $0<n<\omega$ (Theorem \ref{categoricity implies symmetry}).

\begin{theorem}[Adaptation of Theorem 5 of \cite{Va-sym}]\label{symmetry reduced}
Suppose that Hypothesis \ref{hypothesis symmetry section} holds and that Assumption \ref{new assumption} holds.
If $\K$ has symmetry for non-$\mu$-splitting, then for $(\bar M,\bar a,\bar N)\in\K^*_{\mu,\alpha}$ a reduced tower, we can conclude that $\bar M$ is a continuous sequence (i.e. for every limit ordinal $\beta<\alpha$, we have $M_\beta=\Union_{\gamma<\beta}M_\gamma$).

\end{theorem}

\begin{proof}
Suppose $\K$ has symmetry for non-$\mu$-splitting, but reduced towers are not necessarily continuous.
 Let $(\bar M,\bar a,\bar N)\in\K^*_{\mu,\alpha}$ be a discontinuous reduced tower  in $\C$ of minimal length, $\alpha$.  Notice that by Fact \ref{monotonicity of towers}, we can conclude that $\alpha=\delta+1$ for some limit ordinal $\delta$ and that the failure of continuity must occur at $\delta$.
Let $b\in M_\delta\backslash \Union_{\gamma<\delta}M_\gamma$ witness the discontinuity of the tower.  By Assumption \ref{new assumption}, $\Union_{\gamma<\delta}M_\gamma$ must be an amalgamation base.

By the minimality of $\alpha$ and the density of reduced towers (Fact \ref{density of reduced} and Fact \ref{union of reduced is reduced}) we can construct a $<$-increasing and continuous chain of reduced, continuous towers $\langle \T^i=(\bar M,\bar a,\bar N)^i\in\K^*_{\mu,\delta}\mid i<\delta\rangle$ with $(\bar M,\bar a,\bar N)^0:=(\bar M,\bar a,\bar N)\restriction \delta$ inside $\C$. By $\delta$-applications of Fact \ref{density of reduced} inbetween successor stages of the construction we can require that for $\beta<\delta$
\begin{equation}\label{limit at successor}
M^{i+1}_{\beta}\text{ is a }(\mu,\delta)\text{-limit over }N_{\beta}.
\end{equation}
Let $\displaystyle{M^\delta_\delta:=\Union_{i<\delta,\;\beta<\delta}M^i_\beta}$.  See Figure \ref{fig:Mdeltas}.

\begin{figure}[h]
\begin{tikzpicture}[rounded corners=5mm,scale =2.9,inner sep=.5mm]
\draw (0,1.5) rectangle (.75,.5);
\draw (0,1.5) rectangle (1.75,1);
\draw (.25,.75) node {$N_0$};
\draw (1.25,1.25) node {$N_\beta$};
\draw (0,0) rectangle (4,1.5);
\draw (0,1.5) rectangle (3.5,-2);
\draw (.85,.25) node {$M_0$};
\draw(1.25,.25) node {$M_1$};
\draw (1.75,.25) node {$\dots M_\beta$};
\draw (2.3,.25) node {$M_{\beta+1}$};
\draw (3.15,.2) node {$\dots\displaystyle{\Union_{\gamma<\delta}M_\gamma}$};
\draw (3.85, .25) node {$M_\delta$};
\draw (-.5,.25) node {$(\bar M,\bar a,\bar N)$};
\draw (0,1.5) rectangle (3.5, -.4);
\draw (.85,-.15) node {$M^{1}_0$};
\draw (1.75,-.15) node {$\dots M^{1}_\beta$};
\draw (2.3,-.15) node {$M^{1}_{\beta+1}$};
\draw(1.25,-.15) node {$M^{1}_1$};
\draw (3.15,-.2) node {$\dots\displaystyle{\Union_{\gamma<\delta}M^{1}_\gamma}$};
\draw (-.5,-.15) node {$(\bar M,\bar a,\bar N)^1$};
\draw (.85,-.6) node {$\vdots$};
\draw (1.75,-.6) node {$\vdots$};
\draw (2.35,-.6) node {$\vdots$};
\draw (3.2,-.6) node {$\vdots$};
\draw (0,1.5) rectangle (3.5, -1);
\draw (.85,-.85) node {$M^{j}_0$};
\draw (1.75,-.85) node {$\tiny{\dots} M^{j}_\beta$};
\draw (2.3,-.85) node {$M^{j}_{\beta+1}$};
\draw (3,-.85) node {$\dots\Union_{\gamma<\delta}M^{j}_\gamma$};
\draw (-.5,-.85) node {$(\bar M,\bar a,\bar N)^j$};
\draw (0,1.5) rectangle (3.5, -1.35);
\draw (0,1.5) rectangle (1,-2);
\draw(0,1.5) rectangle (1.5, -2);
\draw (0,1.5) rectangle (2.5, -2);
\draw (0,1.5) rectangle (2,-2);
\draw (0,1.5) rectangle (3.5, -2);
\draw (.8,-1.15) node {$M^{j+1}_0$};
\draw (1.8,-1.15) node {$ M^{j+1}_\beta$};
\draw (2.3,-1.15) node {$M^{j+1}_{\beta+1}$};
\draw (3,-1.2) node {$\dots\Union_{\gamma<\delta}M^{j+1}_\gamma$};
\draw (-.5,-1.15) node {$(\bar M,\bar a,\bar N)^{j+1}$};
\draw (.85,-1.6) node {$\vdots$};
\draw (1.75,-1.6) node {$\vdots$};
\draw (2.35,-1.6) node {$\vdots$};
\draw (3.2,-1.6) node {$\vdots$};
\node at (3.75,.75)[circle, fill, draw, label=90:$b$] {};
\node at (2.25,.75)[circle, fill, draw, label=290:$a_\beta$] {};
\node at (1.1,.75)[circle, fill, draw, label=290:$a_1$] {};
\draw (3.75,-1.75) node {$M^\delta_\delta$};
\end{tikzpicture}
\caption{$(\bar M,\bar a,\bar N)$ and the  towers $(\bar M,\bar a,\bar N)^j$ extending $(\bar M,\bar a,\bar N)\restriction\delta$ inside $\C$.} \label{fig:Mdeltas}
\end{figure}

There are two cases: $1)$ we have $b\in M^\delta_\delta$ and $2)$ we have $b\notin M^\delta_\delta$.  
If $b\in M^\delta_\delta$, then we will have found an extension of $(\bar M,\bar a,\bar N)\restriction\delta$  containing $b$ (namely $(\bar M,\bar a,\bar N)^\delta)$) which can easily be lengthened to a discontinuous extension of the entire $(\bar M,\bar a,\bar N)$ tower by taking the $\delta^{th}$ model to be some extension of $M^\delta_\delta$ which is also universal over  $M_\delta$.  This is possible because we have constructed $M^\delta_\delta$ so that it lies in $\C$ along with $M_\delta$.
This discontinuous extension of  $(\bar M,\bar a,\bar N)$ along with $b$ witness that $(\bar M,\bar a,\bar N)$ cannot be reduced.

So suppose that $b\notin M^\delta_\delta$.  Since $M^\delta_\delta$ is a limit model and hence an amalgamation base, we can consider the non-algebraic type 
$\tp(b/M^\delta_\delta)$.
By the $\mu$-superstability assumption, there exists $i^*<\alpha$ so that $\tp(b/M^\delta_\delta)$ does not $\mu$-split over $M^{i^*}_{i^*}$.  By monotonicity of non-splitting, we may assume that $i^*$ is a successor and thus by $(\ref{limit at successor})$, $M^{i^*}_{i^*}$ is a $(\mu,\delta)$-limit over $N_{i^*}$.
Now, referring to the Figure \ref{fig:sym}, apply symmetry to $a_{i^*}$ standing in for $a$, $M^{i^*}_{i^*}$ representing  $M_0$,  $N_{i^*}$ as $N$, $M^\delta_\delta$ as $M$, and $b$ as itself.  We can conclude that there exists  $M^b$ containing $b$, a limit model over $M^{i^*}_{i^*}$, for which $tp(a_{i^*}/M^b)$ does not $\mu$-split over $N_{i^*}$.  

Our next step is to consider the tower formed by the diagonal elements in Figure \ref{fig:Mdeltas}.  In particular let $\T^{diag}$ be the 
tower in $\K^*_{\mu,\delta}$ extending $\T\restriction \delta$ whose models are $M^i_i$ for each $i<\delta$.

Define the tower $\T^b\in\K^*_{\mu,i^*+2}$ by the sequences $\bar a\restriction (i^*+1)$, $\bar N\restriction (i^*+1)$ and $\bar M'$ with $M'_j:=M^j_j$ for $j\leq i^*$ and $M'_{i^*+1}:=M^b$.  Notice that $\T^b$ is an extension of $\T^{diag}\restriction(i^*+2)$ containing $b$.  We will explain how we can use this tower to find a tower $\mathring\T^\delta\in\K^*_{\mu,\delta}$ extending $\T^{diag}$ with $b\in \Union_{j<\delta}\mathring M^\delta_{j}$.  This will be enough to contradict our assumption that $\T$ was reduced.

We define  $\langle \mathring\T^j, f_{j,k}\mid i^*+2\leq j\leq k\leq\delta\rangle$ a directed system of towers so that for $j \geq i^*+2$
\begin{enumerate}
\item\label{base} $\mathring\T^{i^*+2}=\T^b$
\item for $j\leq\delta$, $\mathring\T^j\in\K^*_{\mu,j}$ and lies in $\C$
\item $\T^{diag}\restriction j \leq\mathring\T^j$ for $j\leq\delta$
\item $f_{j,k}(\mathring\T^j)\leq\mathring\T^k\restriction j$ for $j\leq k<\delta$
\item\label{id condition} $f_{j,k}\restriction M^{j}_j=id_{M^{j}_j}$ $j\leq k<\delta$
\item\label{limit M'} $\mathring M^{j+1}_{j+1}$ is universal over $f_{j,j+1}(\mathring M^j_j)$ for $j<\delta$
\item\label{b in} $b\in\mathring M^{j}_{j}$ for $j\leq\delta$
\item\label{non splitting} $\tp(f_{j,k}(b)/M^{k}_{k})$ does not $\mu$-split over $M^{i^*}_{i^*}$ for $j<k<\delta$.
\end{enumerate}

We will define this directed system by induction on $k$, with $i^*+2\leq k\leq\alpha$.  The base case $i^*+2$ is determined by condition \ref{base}.  To cover the successor case, suppose that 
$k=j+1$.  
By our choice of $i^*$, we have $\tp(b/\Union_{l<\alpha}M^{l}_l)$ does not $\mu$-split over $M^{i^*}_{i^*}$. 
So in particular by monotonicity of non-splitting, we notice:
\begin{equation}\label{Mjj non-split}
\tp(b/M^{j+1}_{j+1})\text{ does not }\mu\text{-split over }M^{i^*}_{i^*}.
\end{equation} 
 Using the definition of towers, the choice of $i^*$, and the fact that $M^{j+1}_{j+1}$ was chosen to be a $(\mu,\delta)$-limit over $N_{j+1}$, we can apply symmetry to $a_{j+1}$, $M^{j+1}_{j+1}$, $ \Union_{l<\delta}M^{l}_l$, $b$ and $N_{j+1}$ which will yield $M^b_{j+1}$ a  limit model over $M^{j+1}_{j+1}$ containing $b$ 
 so that $\tp(a_{j+1}/M^b_{j+1})$ does not $\mu$-split over $N_{j+1}$ (see Figure \ref{fig:successor}).

\begin{figure}[h]
\begin{tikzpicture}[rounded corners=5mm, scale=3,inner sep=.5mm]
\draw (0,1.25) rectangle (.75,.5);
\draw (.25,.75) node {$N_{j+1}$};
\draw (0,0) rectangle (3,1.25);
\draw (0,1.25) rectangle (1,0);
\draw (.8,.25) node {$M^{j+1}_{j+1}$};
\draw (3.35, .25) node {$\Union_{l<\delta}M^{l}_l$};
\draw[color=gray] (0,1.25) rectangle (1.5, -.5);
\node at (1.1,-.25)[circle, fill, draw, label=45:$b$] {};
\node at (2,.75)[circle, fill, draw, label=45:$a_{j+1}$] {};
\draw[color=gray] (1.75,-.25) node {$M^{b}_{j+1}$};
\end{tikzpicture}
\caption{A diagram of the application of symmetry in the successor stage of the directed system construction in the proof of Theorem \ref{reduced are continuous}. We have $\tp(b/ \Union_{l<\delta}M^{l}_l)$ does not $\mu$-split over $M^{j+1}_{j+1}$ and $\tp(a_{j+1}/M^{j+1}_{j+1})$ does not $\mu$-split over $N_{j+1}$.  Symmetry implies the existence of  $M^b$ a limit model over $M^{j+1}_{j+1}$ so that $\tp(a_{j+1}/M^b)$  does not $\mu$-split over $N_{j+1}$.} \label{fig:successor}
\end{figure}
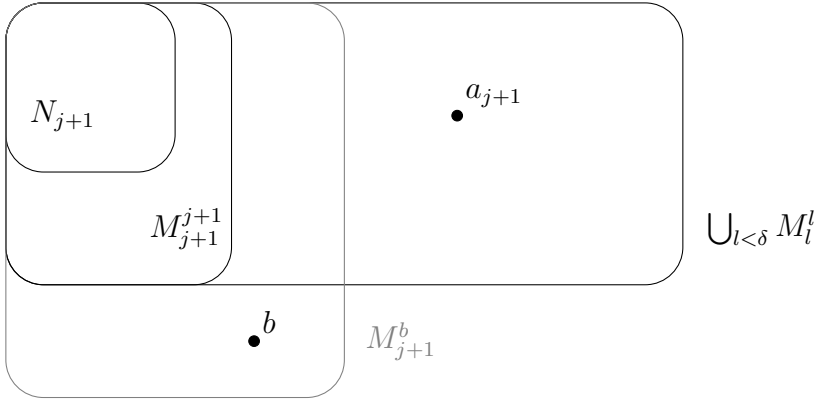

Fix $M'$ to be a model of cardinality $\mu$ extending  both $\mathring M^j_j$ and $M^{j+1}_{j+1}$
Since $M^b_{j+1}$ is a limit model over $M^{j+1}_{j+1}$, there exits $f_{j,j+1}:M'\rightarrow M^b_{j+1}$ with $f_{j,j+1}=id_{M^{j+1}_{j+1}}$ so that $M^b_{j+1}$ is also universal over $f_{j,j+1}(\mathring M^j_j)$.  Notice that condition \ref{non splitting} of the construction is satisfied because of (\ref{Mjj non-split}), invariance, and our choice of $f_{j,j+1}\restriction M_{j+1}^{j+1}=\id$.
Therefore, it is easy to check that $\mathring \T^{j+1}$ defined by the models $\mathring M^{j+1}_l:=f_{j,j+1}(\mathring M^j_l)$ for $l\leq j$ and $\mathring M^{j+1}_{j+1}:=M^b_{j+1}$ are as required.
Then the rest of the directed system can be defined by the induction hypothesis and the mappings $f_{l,j+1}:=f_{l,j}\circ f_{j,j+1}$ for $i^*+2\leq l<j$.

Now consider the limit stage  $k$ of the construction.  First,  let $\grave \T^k$ and $\langle\grave f_{j,k}\mid i^*+2\leq j<k\rangle$ be a direct limit of the system defined so far.  We use the $\grave{}$ notation since these are only approximations to the tower and mappings that we are looking for.  We will have to take some care to find a direct limit that contains $b$ in order to satisfy Condition \ref{b in} of the construction.
By Assumption \ref{new assumption},  our induction hypothesis, and Fact \ref{direct limit lemma}, we may choose this direct limit to lie in $\C$ so that for all $j<k$
\begin{equation*}
\grave f_{j,k}\restriction M^{j}_j=id_{M^{j}_j}.
\end{equation*}
Consequently $\grave M^\alpha_j:=\grave f_{j,k}(\mathring M^j_j)$ is universal over $M^{j}_j$, and $\Union_{j<k}\mathring M^k_j$ is a limit model witnessed by condition \ref{limit M'} of the construction.  Additionally, because $\T^{diag}\restriction k$  is continuous,
 the tower  $\grave\T^k$   composed of the models $\grave M^k_j$, extends $\T^{diag}\restriction k$.

We will next show that for every $j<k$,
\begin{equation}\label{limit non split eqn}
\tp(\grave f_{i^*+2,k}(b)/M^j_j)\text{ does not }\mu\text{-split over }M^{i^*}_{i^*}.
\end{equation}
To see this, recall that for every $j<k$, by the definition of a direct limit, $\grave f_{i^*+2,k}(b)=\grave f_{j,k}(f_{i^*+2,j}(b))$.
By condition \ref{non splitting} of the construction, we know
\begin{equation*}
\tp(f_{i^*+2,j}(b)/M^{j}_{j})\text{ does not }\mu\text{-split over }M^{i^*}_{i^*}.
\end{equation*}
Applying $\grave f_{j,k}$ to this implies $\tp(\grave f_{i^*+2,k}(b)/M^j_j)$ does not $\mu$-split over $M^{i^*}_{i^*}$, establishing $(\ref{limit non split eqn})$.

Because $M^{j+1}_{j+1}$ is universal over $M^j_j$ by construction, we can apply
our assumption of $\mu$-superstability to $(\ref{limit non split eqn})$ yielding
\begin{equation}\label{grave f}
\tp(\grave f_{i^*+2,k}(b)/\Union_{j<k}M^j_j)\text{ does not }\mu\text{-split over }M^{i^*}_{i^*}.
\end{equation}

Because $\grave f_{i^*+2,k}$ fixes $M^{i^*+1}_{i^*+1}$,  $\tp(b/M^{i^*+1}_{i^*+1})=\tp(\grave f_{i^*+2,k}(b)/M^{i^*+1}_{i^*+1})$.
We can then apply the uniqueness of non-splitting extensions to $(\ref{grave f})$ to see that  $\tp(\grave f_{i^*+2,k}(b)/\Union_{j<k}M^j_j)=\tp(b/\Union_{j<k}M^j_j)$.  Thus we can fix $g$ an automorphism of $\C$ fixing $\Union_{j<k}M^j_j$ so that $g(\grave f_{i^*+2,k}(b))=b.$

We will then define $\mathring \T^k$ to be the tower $g(\grave\T^k)$ and the mappings for our directed system will be $f_{j,k}:=g\circ\grave f_{j,k}$ for all $ i^*+2\leq j<k$.
This completes the construction.

Now that we have $\mathring\T^\delta$ a tower extending $\T\restriction\delta$ which contains $b$, we are in a situation 
similar to the proof in case $1)$.  
To contradict that $\T$ is reduced, 
we need only 
lengthen $\mathring\T^\delta$ to a discontinuous extension of the entire $(\bar M,\bar a,\bar N)$ tower by taking the $\delta^{th}$ model to be some extension of $\Union_{i<\delta}\mathring M^i_i$ which is also universal over  $M_\delta$.  This is possible because all the models lie in $\C$.  This discontinuous extension of  $(\bar M,\bar a,\bar N)$ along with $b$ witness that $(\bar M,\bar a,\bar N)$ cannot be reduced.

\end{proof}

Next we adapt the proof of Theorem 5 of \cite{Va-sym} to prove the converse of Theorem \ref{symmetry reduced}.
\begin{theorem}[Adaptation of Theorem 5 of \cite{Va-sym}]\label{converse}
Suppose that Hypothesis \ref{hypothesis symmetry section} holds.
If  every $(\bar M,\bar a,\bar N)\in\K^*_{\mu,\alpha}$ reduced tower is  continuous (i.e. for every limit ordinal $\beta<\alpha$, we have $M_\beta=\Union_{i<\beta}M_i$), then $\K$ has symmetry for non-$\mu$-splitting.

\end{theorem}
\begin{proof}
Suppose that $M$ is an amalgamation base and universal over $M_0$ and that $M_0$ is a limit model over $N$ and so that all these models lie in $\C$.
 Fix $b$ so that the non-algebraic $\tp(b/M)$ does not $\mu$-split over $N$ with $b\in\C$.  Fix $a\in M\backslash M_0$. Without loss of generality, by monotonicity of non-splitting, we may assume that $M$ is a limit model over $M_0$.  Let $\langle M_i\mid i<\delta\rangle$ witness this. We can arrange that $M_{i+1}$ is a limit model over $M_i$ and $a\in M_1$.  To prove $\mu$-symmetry, we will find $M^b$ a limit model over $M_0$ containing $b$ and extending $N$ so that $\tp(a/M^b)$ does not $\mu$-split over $N$.  

We start by  building a tower of length $\delta+1$.  We'll use the models in the sequence $\langle M_i\mid i<\delta\rangle$ as the first part of the tower and we'll define $M_{\delta}$ to be some limit model extending $M$ containing $b$.  
We will set $a_0:=a$ and for $0<i<\delta$ we can choose $a_i\in M_{i+1}\backslash M_i$ realizing the extension of $\tp(a/M_0)$ to $M_i$ that does not $\mu$-split over $N$.
Then set  $N_i:=N$ for each $i$.  Refer to the tower of length $\delta+1$ defined this way as $\mathcal T$. 

Notice that $\mathcal T$ is discontinuous at $\delta$; therefore by our assumption, it is not reduced.  
However at this place of discontinuity, $\Union_{i<\delta}M_i$ is a limit model and hence an amalgamation base.  Therefore $\mathcal T$ is amalgamable.
By the $\mu$-superstability assumptions, our assumption that reduced towers are continuous,  and Fact \ref{density of reduced}, we can find $\mathcal T'$ in $\C$ extending $\mathcal T$ that is reduced, and continuous.  By the continuity of this tower, since $b$ appears in the tower, there exists $j<\delta$ so that $b\in M'_{j}$.  Fix the minimal such $j$ and denote it by $j^*$.  There are two cases to consider

Case 1:  $j^*=0$.  By definition of the ordering on towers, since $\mathcal T<\mathcal T'$, we know that $\tp(a_{0}/M'_{0})$ does not $\mu$-split over $N$.  Thus $M'_{0}$ witnesses $\mu$-symmetry.

Case 2: $j^*>0$.  
By the choice of $a_j$ and uniqueness of non-splitting extensions, we know $\tp(a_{0}/M'_{0})=\tp(a_{j^*}/M'_{0})$.  Thus, there exists $f\in\Aut_{M_0}(\C)$ with $f(a_{j^*})=a_{0}$.  Since $M_1$ is universal over $M_0$, we can also require that our choice of $f$ has the property that $f\restriction M:M \rightarrow_{M_0} M_{1}$.
Because $\tp(b/M)$ does not $\mu$-split over $N$, we know 
\begin{equation*}\label{b non-split}
\tp(f(b)/f(M))=\tp(b/f(M)).
\end{equation*}
This implies there exists an automorphism $g$ of $\C$ fixing $f(M)$ so that $g(f(b))=b$.  

We claim that $M^b:=g(f(M'_{j^*}))$ is as required.  First notice that $b\in M^b$ since $f(b)\in f(M'_{j^*})$ and $g(f(b))=b$.  Next we need to check that $\tp(a_{0}/M^b)$ does not $\mu$-split over $N$.  By the definition of towers, 
\begin{equation*}\label{non-split j^*} \tp(a_{j^*}/M'_{j^*})\text{ does not }\mu\text{-split over }N_{j^*}(=N).  \end{equation*}
By invariance and by our choice of $f$ and $g$ fixing $N$ with $g(f(M'_{j^*}))=M^b$, we can conclude that  
\begin{equation*}
\tp(g(f(a_{j^*}))/M^b)\text{ does not }\mu\text{-split over }N.
\end{equation*} 
By our choice of $f$ taking $a_{j^*}$ to $a_{0}$, we get 
\begin{equation}\label{non-split gf} \tp(g(a_{0})/M^b)\text{ does not }\mu\text{-split over }N.  \end{equation}
Because $g$ fixes $f(M)$ and $a_{0}=f(a_{j^*})\in f(M)$, $(\ref{non-split gf})$ implies that $\tp(a_{0}/M^b)$ does not $\mu$-split over $N$ as required.

\end{proof}

Combining Theorem \ref{converse} with  Theorem 2 of \cite{Va-errata}, 
we  conclude

\begin{corollary}\label{categoricity symmetry down one}
Under Hypothesis \ref{hypothesis symmetry section}, categoricity in $\mu^+$ implies $\mu$-symmetry. 
\end{corollary}

\begin{proof}
Assumption \ref{new assumption}  implies that all towers are nice.  Theorem 2 of of \cite{Va-errata} states that all reduced nice towers of cardinality $\mu$ are continuous provided that the class is categorical in $\mu^+$.  Then Theorem \ref{converse} gives us $\mu$-symmetry.
\end{proof}


Now that we have symmetry in $\lambda$ from categoricity in $\lambda^+$ we can adapt the proof of Corollary 18 of \cite{Va-union} to transfer symmetry from $\lambda$ down to $\mu$ where $\mu^{+n}=\lambda$ for some $1<n<\omega$ in this context in which the full amalgamation property is not assumed.  To transfer symmetry even further down past a limit cardinal we will need to adapt the proof of the Theorem 1.1 of \cite{VV-transfer} which appears in an upcoming paper.

\begin{theorem}\label{categoricity implies symmetry}
Under Hypothesis \ref{hypothesis symmetry section}, categoricity in $\mu^{+n}$ for some $0<n<\omega$ implies $\mu$-symmetry for all $\mu\geq\LS(\K)$.
\end{theorem}

\begin{proof}
By Corollary \ref{categoricity symmetry down one} it is enough to show that $\mu^+$-symmetry implies $\mu$-symmetry.
This is an adaptation of the proof of Theorem 2 of \cite{Va-union}.  
Suppose $\K$ does not have symmetry for $\mu$-non-splitting.  By Theorem \ref{converse} and Hypothesis \ref{hypothesis symmetry section}, $\K$ has a reduced discontinuous tower.  Let $\alpha$ be the minimal ordinal such that $\K$ has a reduced discontinuous tower of length $\alpha$.  By Fact \ref{monotonicity of towers}, we may assume that $\alpha=\delta+1$ for some limit ordinal $\delta$.  Fix $\T=(\bar M,\bar a,\bar N)\in\K^*_{\mu,\alpha}$ a reduced discontinuous tower  with $b\in M_\delta\backslash \Union_{\beta<\delta}M_\beta$.  By Fact \ref{density of reduced}, Fact \ref{union of reduced is reduced}, and the minimality of $\alpha$, we can build an increasing and continuous chain of reduced, continuous towers $\langle\T^i\mid i<\mu^+\rangle$ extending $\T\restriction\delta$ in $\C$.

For each $\beta<\delta$, set $M^{\mu^+}_\beta:=\Union_{i<\mu^+}M^i_\beta$.  
Notice that for each $\beta<\delta$ 
\begin{equation}\label{non-split}
\tp(a_\beta/M^{\mu^+}_\beta)\text{ does not }\mu\text{-split over }N_\beta.
\end{equation}
 If $\tp(a_\beta/M^{\mu^+}_\beta)$ did $\mu$-split over $N_\beta$, it would be witnessed by models inside some $M^i_\beta$, contradicting the fact that $\tp(a_\beta/M^{i}_\beta)$ does not $\mu$-split over $N_\beta$.

We will construct a tower in $\K^*_{\mu^+,\delta}$ from $\bar M^{\mu^+}$.    Notice that by construction, each $M^{\mu^+}_\beta$ is a $(\mu,\mu^+)$-limit model. By Hypothesis \ref{hypothesis symmetry section}.\ref{limits exist hyp}, 
there is a $(\mu^+,\mu^+)$-limit model; so we can apply Proposition \ref{unique saturated proposition} to notice that each $M^{\mu^+}_\beta$ can be represented as a $(\mu^+,\mu^+)$-limit model.  Fix $\langle \grave M^i_\beta\mid i<\mu^+\rangle$ witnessing that $M^{\mu^+}_\beta$ is a $(\mu^+,\mu^+)$-limit model.  Without loss of generality we can assume that $N_\beta\prec_{\K}\grave M^0_\beta$.  
By $\mu^+$-superstability we know that for each $\beta<\delta$ there is $i(\beta)<\mu^+$ so that $\tp(a_\beta/M^{\mu^+}_\beta)$ does not $\mu^+$-split over $\grave M^{i(\beta)}_\beta$.  Set $N^{\mu^+}_\beta:=\grave M^{i(\beta)}_\beta$.  Notice that $(\bar M^{\mu^+},\bar a,\bar N^{\mu^+})$ is a tower in $\K^*_{\mu^+,\delta}$ that lies in $\C$.  Extend $(\bar M^{\mu^+},\bar a,\bar N^{\mu^+})$ to a tower $\T^{\mu^+}\in\K^*_{\mu^+,\alpha}$ by appending to $\bar M^{\mu^+}$ a $\mu^+$-limit model universal over $M_\delta$ which contains $\Union_{\beta<\delta}M^{\mu^+}_\beta$.  This is possible since all of these models lie in $\C$.
Since $\T^{\mu^+}$ is discontinuous, by Theorem \ref{symmetry reduced} and our $\mu^+$-symmetry assumption, we know that it is not reduced.

However, by Hypothesis \ref{hypothesis symmetry section}, our $\mu^+$-symmetry assumption, Theorem \ref{symmetry reduced} and Fact \ref{density of reduced} imply that there exists a reduced, continuous tower $\T^*\in\K^*_{\mu^+,\alpha}$ extending $\T^{\mu^+}$ in $\C$.  By multiple applications of Fact \ref{density of reduced}, we may assume that in $\T^*$ each $M^*_\beta$ is a $(\mu^+,\mu^+)$-limit over $M^{\mu^+}_\beta$.  See Fig. \ref{fig:tower}.

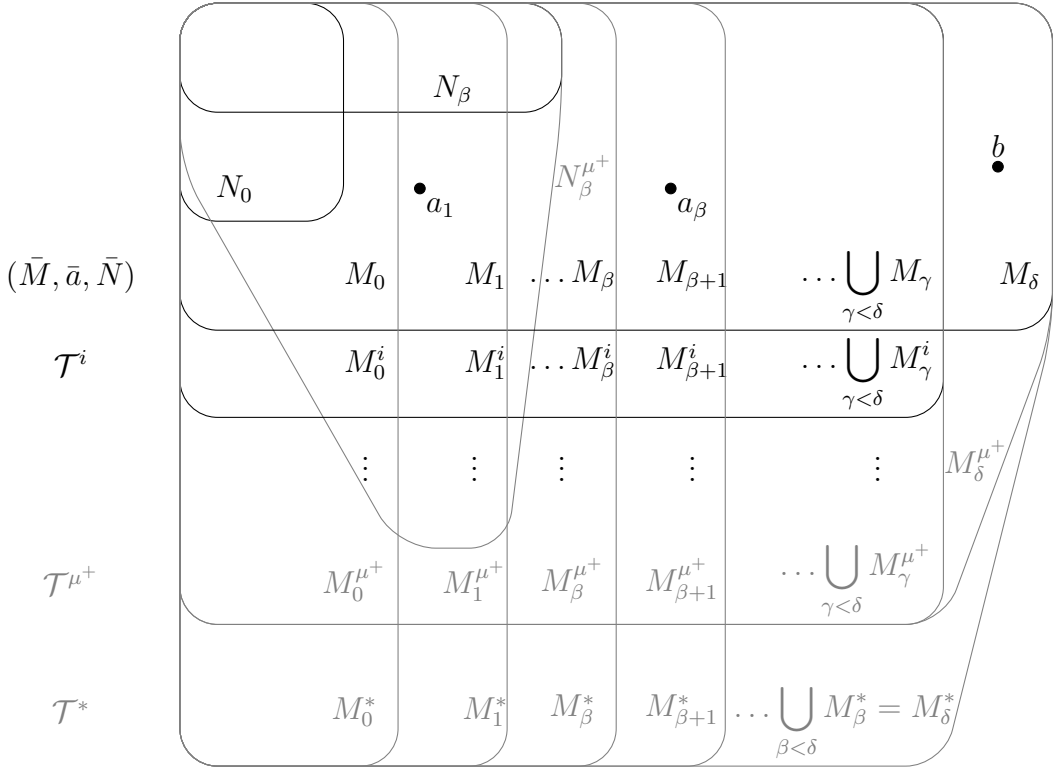
\begin{figure}[h]
\begin{tikzpicture}[rounded corners=5mm,scale =2.9,inner sep=.5mm]
\draw (0,1.5) rectangle (.75,.5);
\draw (0,1.5) rectangle (1.75,1);
\draw (.25,.65) node {$N_0$};
\draw (1.25,1.1) node {$N_\beta$};
\draw[rounded corners=5mm, color=gray]  (0, 1.5) --(0,.75)-- (1,-1) -- (1.5,-1)  -- (1.75,1)--(1.75,1.5)--  cycle;
\draw[color=gray] (1.85,.7) node {$N^{\mu^+}_\beta$};
\draw (0,0) rectangle (4,1.5);
\draw (.85,.25) node {$M_0$};
\draw(1.4,.25) node {$M_1$};
\draw (1.8,.25) node {$\dots M_\beta$};
\draw (2.35,.25) node {$M_{\beta+1}$};
\draw (3.15,.2) node {$\dots\displaystyle{\Union_{\gamma<\delta}M_\gamma}$};
\draw (3.85, .25) node {$M_\delta$};
\draw (-.5,.25) node {$(\bar M,\bar a,\bar N)$};
\draw (0,1.5) rectangle (3.5, -.4);
\draw[rounded corners=5mm, color=gray]  (0, 1.5) -- (0,-2) -- (3.5,-2)  -- (4,0)--(4,1.5) --  cycle;
\draw[rounded corners=5mm, color=gray]  (0, 1.5) -- (0,-1.35) -- (3.5,-1.35)  -- (4,0)--(4,1.5) --  cycle;
\draw (.85,-.15) node {$M^{i}_0$};
\draw (1.8,-.15) node {$\dots M^{i}_\beta$};
\draw (2.35,-.15) node {$M^{i}_{\beta+1}$};
\draw(1.4,-.15) node {$M^{i}_1$};
\draw (3.15,-.2) node {$\dots\displaystyle{\Union_{\gamma<\delta}M^{i}_\gamma}$};
\draw (-.5,-.15) node {$\T^i$};
\draw (.85,-.6) node {$\vdots$};
\draw (1.75,-.6) node {$\vdots$};
\draw (2.35,-.6) node {$\vdots$};
\draw (3.2,-.6) node {$\vdots$};
\draw (1.35,-.6) node {$\vdots$};
\draw[color=gray] (0,1.5) rectangle (3.5, -1.35);
\draw[color=gray] (0,1.5) rectangle (1,-2);
\draw[color=gray](0,1.5) rectangle (1.5, -2);
\draw[color=gray] (0,1.5) rectangle (2.5, -2);
\draw[color=gray] (0,1.5) rectangle (2,-2);
\draw[color=gray] (.8,-1.15) node {$M^{\mu^+}_0$};
\draw[color=gray] (1.8,-1.15) node {$ M^{\mu^+}_\beta$};
\draw[color=gray] (2.3,-1.15) node {$M^{\mu^+}_{\beta+1}$};
\draw[color=gray](1.35,-1.15) node {$M^{\mu^+}_1$};
\draw[color=gray] (3.1,-1.15) node {$\dots\displaystyle{\Union_{\gamma<\delta}M^{\mu^+}_\gamma}$};
\draw[color=gray] (-.5,-1.15) node {$\T^{\mu^+}$};
\draw[color=gray] (-.5,-1.75) node {$\T^{*}$};
\draw[color=gray] (.8,-1.75) node {$M^{*}_0$};
\draw[color=gray](1.4,-1.75) node {$M^*_1$};
\draw[color=gray] (1.8,-1.75) node {$ M^{*}_\beta$};
\draw[color=gray] (2.3,-1.75) node {$M^{*}_{\beta+1}$};
\node at (3.75,.75)[circle, fill, draw, label=90:$b$] {};
\node at (2.25,.65)[circle, fill, draw, label=290:$a_\beta$] {};
\node at (1.1,.65)[circle, fill, draw, label=290:$a_1$] {};
\draw[color=gray] (3.65, -.6) node {$M^{\mu^+}_\delta$};
\draw[color=gray] (3.05,-1.8) node {$\dots\displaystyle{\Union_{\beta<\delta}M^*_\beta}=M^*_\delta$};
\end{tikzpicture}
\caption{The  towers in the proof of Theorem \ref{categoricity implies symmetry}.  The towers composed of models of cardinality $\mu$ are  black and the towers composed of models of cardinality $\mu^+$ are gray.} \label{fig:tower}
\end{figure}

\begin{claim}\label{star non-split}
For every $\beta<\alpha$, 
$\tp(a_\beta/M^*_\beta)$ does not $\mu$-split over $N_\beta$.
\end{claim}
\begin{proof}
Since $M^*_\beta$ and $M^{\mu^+}_\beta$ are both $(\mu^+,\mu^+)$-limit models over $N^{\mu^+}_\beta$, there exists $f:M^*_\beta\cong_{N^{\mu^+}_\beta}M^{\mu^+}_\beta$.  Since $\T^*$ is a tower extending $\T^{\mu^+}$, we know that $\tp(a_\beta/M^*_\beta)$ does not $\mu^+$-split over $N^{\mu^+}_\beta$.  Therefore by the definition of non-splitting, it must be the case that $\tp(f(a_\beta)/M^{\mu^+}_\beta)=\tp(a_\beta/M^{\mu^+}_\beta)$.  From this equality of types we can fix $g\in\Aut_{M^{\mu^+}_\beta}(\C)$ with $g(f(a_\beta))=a_\beta$.
An application of $(g\circ f)^{-1}$ to $(\ref{non-split})$ yields the statement of the claim.

\end{proof}

Since $\T^*$ is continuous and extends $\T^{\mu^+}$ which contains $b$, there is $\beta<\delta$ such that $b\in M^*_\beta$.  Fix such a $\beta$.

We now will define a tower $\T^b\in\K^*_{\mu,\alpha}$ extending $\T$.  For $\gamma<\beta$, take $M^b_\gamma:=M_\gamma$.  For $\gamma=\beta$, let $M^b_\gamma$ be a $(\mu,\mu)$-limit model over $M_\gamma$ inside $M^*_\gamma$ so that $b\in M^b_\gamma$.  For $\gamma>\beta$, take $M^b_\gamma$ to be a $(\mu,\mu)$-limit model over $M_\gamma$ so that $\Union_{\xi<\gamma}M^b_{\xi}\prec_{\K}M^b_\gamma$.  Notice that by Claim \ref{star non-split} and monotonicity of non-splitting, the tower $\T^b$ defined as $(\bar M^b,\bar a,\bar N)$ is a tower extending $\T$ with $b\in (M^b_\beta\backslash M_\beta)\bigcap M_\alpha$.  This contradicts our assumption that $\T$ was reduced.

\end{proof}

\section{Proof of Theorem \ref{necessary theorem}}\label{sec:main theorem}

First notice that the assumptions of Theorem \ref{necessary theorem} imply the following properties for every $\kappa$ with $\LS(\K)\leq\kappa<\lambda$:
\begin{enumerate}
\item $\kappa$-superstability \cite[Facts 1.4.7 and 1.48]{Va}. 
\item Limit models of cardinality $\kappa$ are amalgamation bases \cite[Fact 1.3.10]{ShVi}.
\item Density of amalgamation bases of cardinality $\kappa$ \cite[Theorem 1.2.4]{ShVi}.
\item For every amalgamation base $M$ of cardinality $\kappa$ there exists $M'\in\K_\mu$ a limit model over $M$ \cite[Fact 1.3.10]{ShVi}.
\end{enumerate}

\begin{proof}[Proof of $\ref{assumption item}\Rightarrow\ref{limit item}$ of Theorem \ref{necessary theorem}]
This  is the content of \cite{Va} along with Theorem \ref{reduced are continuous}.
\end{proof}

\begin{proof}[Proof of $\ref{limit item}\Rightarrow\ref{limit not over item}$]
Suppose that $M$ is a $(\mu,\theta)$-limit model over $M_0$ and $M'$ is a $(\mu,\theta')$-limit model over $M'_0$ (perhaps of no relation to $M_0$).  By categoricity in $\lambda$ we may assume without loss of generality that there is $N\in\K_\lambda$ so that $M,M'\prec_{\K}N$.  By the Downward L\"{o}wenheim Skolem axiom of AECs, we can find $M^*$ an extension of $M$ of cardinality $\mu$ containing $M'$.  By the coherence axiom, we may assume that $M'\prec_{\K}M^*$ as well.  By the existence of limit models, we can assume that $M^*$ is a $(\mu,\theta')$-limit model over $M$.  Notice that $M^*$ is also a $(\mu,\theta')$-limit model over $M_0$.  By \ref{limit item}, $M^*$ and $M$ are isomorphic over $M_0$.  

Furthermore, notice that $M^*$ is a $(\mu,\theta')$- limit model over $M'$ as well.   Then we also know that $M^*$ is a $(\mu,\theta')$-limit model over $M'_0$.  By a back and forth construction $M^*$ and $M'$ are isomorphic over $M'_0$.  Thus, combining this information with the previous paragraph, we conclude that $M'$ and $M$ are isomorphic.

\end{proof}

\begin{proof}[Proof of $\ref{limit not over item}\Rightarrow\ref{union item}$ of Theorem \ref{necessary theorem}]
This argument is an adaptation of the proof of Theorem 20 of \cite{Va-union}.
Fix $M=\Union_{i<\theta}M_i$ where $\langle M_i\in\K_{\kappa^+}\mid i<\theta\rangle$ is an increasing and continuous sequence of   saturated models dense with amalgamation bases.  Fix $N\prec_{\K}M$ an amalgamation base of cardinality $\kappa$.
Let $p:=\tp(a/N)$.
Suppose for the sake of contradiction that $p$ is not realized in $M$.

We can use the assumption that each $M_i$ is dense with amalgamation bases and the Downward L\"{o}wenheim-Skolem axiom to find 
$\langle N_i\in\K_\kappa\mid i<\theta\rangle $  an increasing and continuous sequence of amalgamation bases so that $N\bigcap M_i\subseteq N_i\prec_{\K}M_i$ for each $i<\theta$.
Because each $M_{i+1}$ is $\kappa^+$-saturated and dense with amalgamation bases, by Lemma \ref{sat universal lemma} we may further select this sequence so  that $N_{i+1}$ is universal over $N_i$.
Notice that $\Union_{i<\theta}N_i$ is a $(\kappa,\theta)$-limit model and hence an amalgamation base.  Because we are assuming that $a\notin M$, we know that $a\notin \Union_{i<\theta}N_i$.  This allows us to assume without loss of generality that $N$ is the $(\kappa,\theta)$-limit model $\Union_{i<\theta}N_i$ and $p:=\tp(a/N)$ is a Galois-type omitted in $M$. 

Then by $\kappa$-superstability, we may assume without loss of generality that $p$ does not $\kappa$-split over $N_0$, by possibly renumbering the sequences $\bar N$ and $\bar M$.  

For each $i<\theta$, because $M_i$ is $\kappa^+$-saturated and dense with amalgamation bases, by Corollary \ref{sat is limit} and Proposition \ref{unique saturated proposition},
$M_i$ is isomorphic to both a $(\kappa,\kappa^+)$-limit model and a $(\kappa^+,\kappa^+)$-limit model.  So, inside each $M_i$ we can find a $(\kappa^+,\kappa^+)$-limit model witnessed by a sequence that we will denote by $\langle\grave M_i^\alpha\in\K_{\kappa^+}\mid \alpha<\kappa^+\rangle$,  and  we may arrange the enumeration so that  $N_i\prec_{\K}\grave M^0_i$.

We will build a directed system  of models $\langle M^*_i\mid i<\theta\rangle$ with mappings $\langle f_{i,j}\mid i\leq j<\theta\rangle$ so that the following conditions are satisfied:
\begin{enumerate}
\item $M^*_i\in\K_{\kappa^+}$.
\item $M^*_i\preceq_{\K}\Union_{\alpha<\kappa^+}\grave M_i^\alpha\preceq_{\K}M_i$. 

\item for $i\leq j<\theta$, $f_{i,j}:M^*_i\rightarrow M^*_j$. 
\item\label{identity condition} for $i\leq j<\theta$, $f_{i,j}\restriction N_i=id_{N_i}$.
\item\label{univ condition direct limit} $M^*_{i+1}$ is universal over $f_{i,i+1}(M^*_i)$.

\end{enumerate}

 Refer to Figure \ref{fig:T*}.
 \begin{figure}[tb]
\begin{tikzpicture}[rounded corners=5mm,scale =2.5,inner sep=.5mm]
\draw (0,0) rectangle (4.5,-.5);
\draw (0,0) rectangle (4.5,-2);
\draw (0,0) rectangle (1,-2);
\draw (0,0) rectangle (2,-2);
\draw (0,0) rectangle (3,-2);
\draw (.85,-.4) node {$N_0$};
\draw (1.75,-.4) node {$\dots N_j$};
\draw (2.85,-.4) node {$N_{j+1}$};
\draw (3.85,-.4) node {$\dots \Union_{i<\theta}N_i=N$};
\draw (.8,-1.9) node {$M_0$};
\draw (1.7,-1.9) node {$\dots M_j$};
\draw (2.75,-1.9) node {$M_{j+1}$};
\draw (3.85,-1.9) node {$\dots \Union_{i<\theta}M_i=M$};
\coordinate (m01-in) at (0,-.9);
\coordinate (m01-out) at (1,-.9);
\draw    (m01-in) to[out=-20,in=200] coordinate[pos=0.7](A1)  (m01-out);
\draw (.6,-.9) node {$M^*_{0}$};
\coordinate (m10-in) at (0,-.5);
\coordinate (m10-out) at (2,-.5);
\draw    (m10-in) to[out=-20,in=240] coordinate[pos=0.7](Ai)  coordinate[pos=.8](ai)(m10-out);
\draw (1.4,-.7) node {$M^*_{j}$};
\coordinate (m1t-in) at (0,-.4);
\coordinate (m1t-out) at (3,-0.5);
\draw    (m1t-in) to[out=-80,in=240] coordinate[pos=0.8](Ai1) (m1t-out);
\draw (2.6,-.7) node {$\grave M^1_{j+1}$};
\draw [->, shorten >=3pt] (A1) to [bend right=65] node[pos=0.7,below] {$f_{0,j}$}(Ai);
\draw [->, shorten >=3pt] (ai) to [bend right=25] node[pos=0.7,above] {$f_{j,j+1}$}(Ai1);
\draw    (m1t-in) to (.2, -1.5) to (2.2,-1.8) to coordinate[pos=0.8](Ai1) (m1t-out);
\draw (1.55,-1.55) node {$\grave M^2_{j+1}=M^*_{j+1}$};
\end{tikzpicture}
\caption{The directed system in the proof of Theorem \ref{necessary theorem}.} \label{fig:T*}
\end{figure}
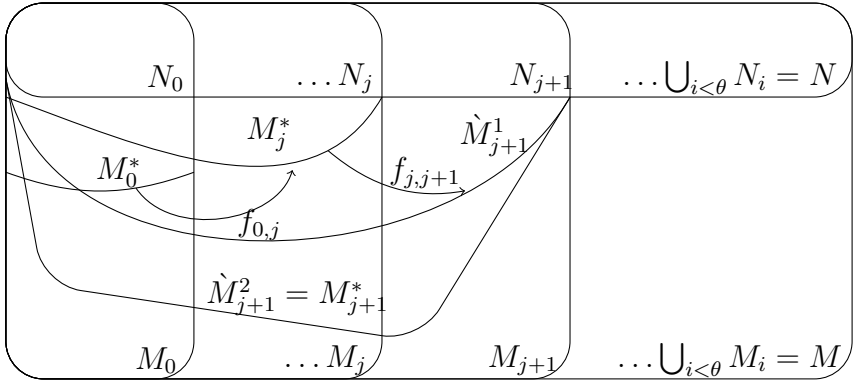

The construction is possible.  Take $M^*_0$ to be $\grave M_0^1$ and $f_{0,0}=\id$.  At limit stages take $M^{**}_i$ and $\langle f^{**}_{k,i}\mid k<i\rangle$ to be a direct limit as in Fact \ref{direct limit lemma} which is possible because each $N_i$ is an amalgamation base.  We do not immediately get that $M^{**}_i\preceq_{\K}M_i$; we just know we can choose $M^{**}_i$ to contain $N_i$ by the continuity of $\bar N$ and condition \ref{identity condition} of the construction.  We also know by condition \ref{univ condition direct limit} that $M^{**}_i$ is a $(\kappa^+,i)$-limit model  witnessed by $\langle f_{k,i}(M^*_k)\mid k<i\rangle$.
By the uniqueness  of limit models of cardinality $\kappa^+$ and Proposition  \ref{unique saturated proposition}, $M^{**}_i$ is a $(\kappa^+,\kappa^+)$-limit model.  Since $N_i$ has cardinality $\kappa$, being able to write $M^{**}_i$ as a $(\kappa^+,\kappa^+)$-limit model tells us that $M^{**}_i$ is $\kappa^+$-universal over $N_i$.  Recall that  $\Union_{\alpha<\kappa^+}\grave M^\alpha_i$ is also a $(\kappa^+,\kappa^+)$-limit model containing $N_i$.  Therefore, by a back-and-forth argument, we can find an isomorphism $g$ from $M^{**}_i$ to $\Union_{\alpha<\kappa^+}\grave M^\alpha_i$ fixing $N_i$.  Now take $M^*_i:=g(M^{**}_i)=\Union_{\alpha<\kappa^+}\grave M^\alpha_i$, $f_{k,i}:=g\circ f^{**}_{k,i}$ for $k<i$, and $f_{i,i}=\id$.  

For the successor stage of the construction, assume that $M^*_j$ and $\langle f_{k,j}\mid k\leq j\rangle$ have been defined.  Since $M^*_j$ is a model of cardinality $\kappa^+$ containing $N_j$ and because $\grave M^{1}_{j+1}$ is $\kappa^+$-universal over $N_{j+1}$ we can find a embedding $g:M^*_j\rightarrow \grave M^1_{j+1}$ with $g\restriction N_j=\id_{N_j}$.  Take $M^*_{j+1}:=\grave M^2_{j+1}$,  set $f_{k,j+1}:=g\circ f_{k,j}$ for all $k\leq j$, and define $f_{j+1,j+1}:=\id$.  This completes the construction.

Take $M^*$ in $\C$ with mappings $\langle f_{i,\theta}\mid i<\theta\rangle$  to be the direct limit of the system as in Fact \ref{direct limit lemma}.  While $M^*$ may not be inside $M$, we can arrange that $f_{i,\theta}\restriction N_i=\id_{N_i}$ and that $N\prec_{\K}M^*$.  Notice that by condition \ref{univ condition direct limit} of the construction, $M^*$ is a $(\kappa^+,\theta)$-limit model.   By the uniqueness of $\kappa^+$-limit models, we know that $M^*$ is saturated.  

For each $i<\theta$, let $f^*_{i,\theta}\in\Aut(\C)$ extend $f_{i,\theta}$ so that $f^*_{i,\theta}(N)\preceq_{\K}M^*$.  This is possible since we know that $M^*$ is $\kappa^+$-universal over $f_{i,\theta}(M_i)$ by condition \ref{univ condition direct limit} of the construction.  
Let $N^*\prec_{\K}M^*$ be a model of cardinality $\kappa$ extending $N$ and $\Union_{i<\theta} f^*_{i,\theta}(N)$.  By the extension property for non-$\kappa$-splitting, we can find $p^*\in\gaS(N^*)$ extending $p$ so that 
\begin{equation}\label{p*}
p^*\text{ does not }\kappa\text{-split over }N_0.
\end{equation}
  Since $M^*$ is a saturated model of cardinality $\kappa^+$,
we can find $b^*\in M^*$ realizing $p^*$.  By the definition of a direct limit, there exists $0<i<\theta$ and $b\in M^*_i$ so that $f_{i,\theta}(b)=b^*$.  

Because $f_{i,\theta}\restriction N_i=id_{N_i}$, we know that $b\models p\restriction N_i$.  Suppose for sake of contradiction that there is some $j>i$ so that $\tp(b/N_j)\neq p\restriction N_j$.  Then, by the uniqueness of non-splitting extensions, it must be the case that $\tp(b/N_j)$ $\kappa$-splits over $N_0$.  By invariance, 
\begin{equation}\label{non-split equation}
\tp(f_{i,\theta}(b)/f^*_{i,\theta}(N_j)) \;\kappa\text{-splits over }N_0.
\end{equation}  By monotonicity of non-splitting,  the definition of $b$, and choice of $N^*$ containing $f^*_{i,\theta}(N)$, $(\ref{non-split equation})$ implies $\tp(b^*/N^*)$ $\kappa$-splits over $N_0$.  This contradicts $(\ref{p*})$.

Since $b\models p\restriction N_j$ for all $j<\theta$ and $p\restriction N_j$ does not $\kappa$-split over $N_0$, $\kappa$-superstability implies that $\tp(b/N)$ does not $\kappa$-split over $N_0$.  By uniqueness of non-$\kappa$-splitting extensions $\tp(b/N)=p$.  Since $b\in M_i$, we are done.

\end{proof}

\begin{proof}[Proof of $\ref{union item}\Rightarrow\ref{assumption item}$ of Theorem \ref{necessary theorem}]

First notice that by 
Lemma \ref{limits are dense with ab} every limit model is dense with amalgamation bases.  Next we show that  by $\ref{union item}$ every limit model of cardinality $\mu=\kappa^+$ is saturated.  To see this consider $N$ a limit model of cardinality $\kappa^+$ witnessed by $\langle N_i\mid i<\theta\rangle$.  By $\kappa^+$-applications of Fact \ref{limits are ab}, for each $N_i$ we can find $N'_i$  a $(\kappa^+,\kappa^+)$-limit model extending $N_i$.  
By Fact \ref{limits are ab} and Proposition \ref{unique saturated proposition}
each $N'_i$ is a $(\kappa,\kappa^+)$-limit model.  Thus each $N'_i$ is
saturated and dense with $\kappa$-amalgamation bases.  
Because $N_{i+1}$ is universal over $N_i$ there is $f_i:N'_i\rightarrow_{N_i} N_{i+1}$.  Let $N^*_i:=f_i(N'_i)$.  Notice that $\langle N^*_i\mid i<\theta\rangle$ is an increasing sequence of saturated models dense with amalgamation bases and  $N=\Union_{i<\theta}N^*_i$.  Thus by our assumption $\ref{union item}$, $N$ is saturated.  

To prove $\ref{assumption item}$, suppose that $\langle M_i\mid i<\theta\rangle$ is an increasing and continuous chain of limit models each of cardinality $\kappa^+$.  By the previous paragraph we can apply $\ref{union item}$ to the sequence $\langle M_i\mid i<\theta\rangle$ to conclude that $M:=\Union_{i<\theta}M_i$ is saturated.  By Lemma \ref{union is dense}, $M$ is dense with amalgamation bases.  
By Corollary \ref{sat is ab},  $M$ is an amalgamation base as required.

\end{proof}

The question remains:  Are the assumptions of Theorem \ref{necessary theorem} enough on their own to prove that the union of an increasing and continuous chain of limit models is an amalgamation base?  This is answered affirmatively in \cite{BVV}.

\section{Acknowledgements}
The author is grateful to Sebastien Vasey for comments and suggestions on earlier drafts of this paper which greatly improved the clarity and accuracy of the presentation.  The author also thanks the organizers of the special session and satellite conference at the 2015 Joint Mathematics Meeting that led to the creation of this volume and to the referee who gave constructive advice on the exposition.





\bibliographystyle{model1-num-names}
\bibliography{<your-bib-database>}

\begin{thebibliography}{00}




\bibitem{GB} Boney, William and Rami Grossberg.
\newblock Forking in short and tame AECs.
\newblock http://arxiv.org/abs/1306.6562


\bibitem{BVV}Boney, William, Monica VanDieren and Sebastien Vasey.
\newblock Some instances of uniqueness of limit models in abstract elementary classes with no maximal models.
\newblock https://arxiv.org/abs/1611.05292

\bibitem{BV-survey}Boney, William and Sebastien Vasey.
\newblock A survey on tame abstract elementary classes.
\newblock http://arxiv.org/abs/1512.00060


\bibitem{Dr} Drueck, Fred.   Limit Models, Superlimit Models, and Two Cardinal Problems in
Abstract Elementary Classes.  Ph.D. Thesis at the University of Illinois at Chicago.  (2013).


\bibitem{Gr1}
Grossberg, Rami.
\newblock Classification theory for non-elementary classes,
\newblock   {\bf Logic and Algebra}, ed. Yi Zhang, Contemporary
Mathematics,  {\bf 302}, (2002) AMS,  pp. 165--204.



\bibitem{GV}Grossberg, Rami and Monica VanDieren.
\newblock Galois-stability of tame abstract elementary classes.  
\emph{Journal of Mathematical Logic}  {\bf 6}.1 (2006)  25--49.

\bibitem{GV2} ---.
Categoricity from one successor cardinal in tame abstract elementary classes.  
 \emph{Journal of Mathematical Logic} {\bf 6}.2 (2006) 181--201.   
%
%
\bibitem{GVV}
Grossberg, Rami, Monica VanDieren and Andr\'{e}s Villaveces.
\newblock Uniqueness of limit models in abstract elementary classes.
\newblock  \emph{Mathematical Logic Quarterly} {\bf 62}.4-5 (2016) 367--382.

\bibitem{KoSh}
Kolman, Oren and  Saharon Shelah.
\newblock Categoricity of theories in $L_{\kappa, \omega}$ when
$\kappa$ is a measurable cardinal.  Part I.
\newblock {\em Fundamentae Mathematicae}, {\bf151} (1996) 209--240.

\bibitem{Lo}
\L o\'{s}, Jerzy. 
\newblock On the categoricity in power of elementary deductive systems and
related problems,
\newblock {\em Colloq. Math.}, {\bf 3} (1954) 58---62.
 




\bibitem{Sh394}
Shelah, Saharon.
\newblock Categoricity of abstract classes with amalgamation,
\newblock \emph{Annals of Pure and Applied Logic}, {\bf 98} (1999) 261--294.


\bibitem{Sh09b}
---.
\newblock \emph{Classification theory for abstract elementary classes 2}, Studies in Logic: Mathematical logic and foundations, vol. 20, College Publications, 2009.


\bibitem{ShVi}
Shelah, Saharon and Andr\'{e}s Villaveces.
\newblock Toward categoricity for classes with no maximal models.
\newblock \emph{Annals of Pure and Applied Logic}, {\bf 97}.1-3 (1999)
1--25.  

\bibitem{Va-thesis}
VanDieren, Monica.
\newblock Categoricity and stability in abstract elementary classes.
\newblock Ph.D. Thesis, Carnegie Mellon University, 2002.

\bibitem{Va}
---.
\newblock Categoricity in abstract elementary classes with no maximal
models.
\emph{Annals of Pure and Applied Logic} {\bf 141} (2006) 108--147. 

\bibitem{Va-errata}
---.
\newblock Erratum to `Categoricity in abstract elementary classes with no maximal models' [Ann. Pure Appl. Logic 141 (2006) 108--147.]  
\newblock \emph{Annals of Pure and Applied Logic}, {\bf 164} (2013) 131--133.  DOI 10.1016/jpal.2012.09.003. 


\bibitem{Va-sym}
---.
\newblock Superstability and symmetry.
\newblock \emph{Annals of Pure and Applied Logic} {\bf 167}.12 (2016) 1171--1183.

\bibitem{Va-union}
---.
\newblock Symmetry and the union of saturated models in superstable abstract elementary classes.
\newblock \emph{Annals of Pure and Applied Logic} {\bf 167} (2016) 395--407.



\bibitem{VV-transfer}
VanDieren, Monica and Sebastien Vasey.
\newblock Symmetry in Abstract Elementary Classes with Amalgamation.
\newblock http://arxiv.org/abs/1508.03252




\bibitem{Vas2}
Vasey, Sebastien.
\newblock Shelah's eventual categoricity conjecture in universal classes: part II.
\newblock To appear in \emph{Selecta Mathematica.}
\newblock{http://arxiv.org/abs/1602.02633}

\bibitem{ViZa2}
Villaveces, Andr\'{e}s and Pedro Zambrano. \newblock Limit Models in
Metric Abstract Elementary Classes: the categorical case. 
\newblock \emph{Mathematical Logic Quarterly} {\bf 62}.4-5. (2016) 319--334.

\bibitem{Za}
Zambrano, Pedro.
\newblock Around Superstability in Metric Abstract Elementary Classes.
\newblock Ph.D. Thesis at the National University of Colombia at Bogot\'{a}.  (2011).




 \end{thebibliography}



\end{document}